\documentclass{article}

\usepackage[utf8]{inputenc}
\usepackage[english]{babel}
\usepackage{csquotes}
\usepackage{amsthm}
\usepackage{amsmath}
\usepackage{graphicx}
\usepackage{amssymb}
\usepackage{amsfonts}
\usepackage{comment}
\usepackage{float}
\usepackage[nottoc]{tocbibind}
\usepackage{caption}
\usepackage{subcaption}
\usepackage{float}
\usepackage[section]{placeins}
\usepackage{setspace}
\usepackage[margin=1.0in]{geometry}
\usepackage{etoolbox}
\theoremstyle{remark}
\usepackage{hyperref}
\usepackage{cleveref}
\usepackage{enumerate}

\usepackage{biblatex} 
\theoremstyle{definition}
\newtheorem{definition}{Definition}[section]

\newtheorem{theorem}[definition]{Theorem}

\newtheorem{lemma}[definition]{Lemma}
\newtheorem{proposition}[definition]{Proposition}
\newtheorem{conjecture}[definition]{Conjecture}

\begin{document}

\title{Polyhedra with hexagonal and triangular faces and  three faces around each vertex}
\author{Linda Green\footnote{University of North Carolina at Chapel Hill, greenl@email.unc.edu}         \and
        Stellen Li\footnote{University of North Carolina at Chapel Hill, stellen@ad.unc.edu}
}

\maketitle
\begin{abstract}
We analyze polyhedra composed of hexagons and triangles with three faces around each vertex, and their 3-regular planar graphs of edges and vertices, which we call ``trihexes''. Trihexes are analogous to fullerenes, which are 3-regular planar graphs whose faces are all hexagons and pentagons. Every trihex can be represented as the quotient of a hexagonal tiling of the plane under a group of isometries generated by $180^\circ$ rotations. Every trihex can also be described with either one or three ``signatures'': triples of numbers $(s, b, f)$ that describe the arrangement of the rotocenters of these  rotations. Simple arithmetic rules relate the three signatures that describe the same trihex. We obtain a bijection between trihexes and equivalence classes of signatures as defined by these rules. Labeling trihexes with signatures allows us to put bounds on the number of trihexes for a given number vertices $v$ in terms of the prime factorization of $v$ and to prove a conjecture concerning trihexes that have no ``belts'' of hexagons. 

\end{abstract}

\section{Introduction}
\label{sec:intro}
Motivated by the study of polyhedra, this paper analyzes 3-regular planar graphs whose faces all have three or six sides. We call these graphs \emph{trihexes}. Trihexes have been analyzed by Deza and Dutour (\cite{deza2005zigzag} and \cite{deza2008geometry}), Gr\"ubaum and Motzkin \cite{grunbaum_motzkin_1963}, and others. We refer to faces with three sides as  ``triangles'' and faces with six sides as ``hexagons'', even though these faces may not have straight edges and  may be unbounded. 

Trihexes are analogous to fullerenes, which are 3-regular planar graphs whose faces all have five or six sides. Fullerenes have received much attention because when viewed as polyhedra, they have physical manifestations as carbon molecules. Fullerenes have been analyzed by Brinkmann, Goedgebeur, and McKay \cite{brinkmann2012generation} and others.

In this paper, Section~\ref{sec:examples} explains how every triple of numbers $(s, b, f)$ with $s \geq 0$, $b \geq 0$, and $0 \leq f \leq s$ (a ``signature'')  describes a unique trihex. The number $s$ gives the number of hexagons that lie in a chain capped by triangles (a ``spine''), $b$ gives the number of rings of hexagons (``belts'') that surround and separate the spines, and $f$ describes the rotation of the two spines relative to each other. Furthermore, every trihex can be described with at least one signature.  Every trihex can also be described as the quotient of a hexagonal tiling of the plane under a group generated by $180^\circ$ rotations, as shown in Section~\ref{sec:grid}. In this context, the signatures $(s, b, f)$ describe the arrangement of the rotocenters of these  rotations. 
 Although there can be three distinct signatures that describe the same trihex, simple arithmetic rules given in Section~\ref{sec:equivalent signatures} relate the signatures that characterize the same trihex. We thus obtain a bijection between trihexes and equivalence classes of signatures as defined by these rules.   In Section~\ref{sec:howMany}, we use our classification of trihexes in terms of signatures  to put bounds on the number of trihexes with $v$ vertices in terms of the prime factorization of $\dfrac{v}{4}$. In Section~\ref{sec:tight} we prove a conjecture about the ``graph of curvatures'' from \cite{deza2005zigzag}.

The results in this paper can be applied to polyhedra whose faces are all triangles and hexagons and have three faces around each vertex, but are not necessarily convex. We will call these polyhedra \emph{trihex polyhedra}. The tetrahedron is a convex trihex polyhedron described by the triple $(0,0,0)$. All other convex trihex polyhedra are described by triples $(s, b, f)$ with $s > 0$, and cannot be described by triples with $s = 0$. Non-convex trihex polyhedra can be described by triples $(s, b, f)$ where $s = 0$ and $b > 0$. The correspondence between trihexes and convex trihex polyhedra follows from Steinitz's theorem \cite{grunbaum1967convex} or \cite{ziegler2012lectures}, as explained in Section~\ref{sec:convexVsNonconvex}.

\section{Definitions and preliminaries}
\label{sec:defs}

\begin{definition}
A \emph{trihex} is a finite, connected, 3-regular planar graph whose faces all have three or six sides. 
\end{definition} 

\begin{definition}
A \emph{polyhedron} is a union of polygons in $R^3$ which is homeomorphic to a sphere.  Any pair of polygons intersect either in the empty set, a vertex, an edge, or a union of vertices and/or edges. 
\end{definition}

\begin{definition}
A \emph{trihex polyhedron} is a polyhedron whose faces all are triangles or hexagons and has three faces around each vertex.
\end{definition}

\begin{definition}
Two polyhedra are \emph{equivalent} if there is an orientation preserving  homeomorphism of the sphere that takes the faces, edges, and vertices of one polyhedron to the faces, edges, and vertices, respectively, of the other. 
\end{definition}

The requirement that the homeomorphism be orientation preserving means that left-handed and right-handed versions of chiral polyhedra are not equivalent. We make the same distinction  for trihexes. By a theorem of Whitney (\cite{whitney19332}, or see \cite{mohar2001graphs})  two planar graphs are isomorphic if and only if there is a homeomorphism of the sphere whose restriction to the planar graph gives a graph isomorphism. We consider trihexes equivalent if and only if an orientation-preserving homeomorphism can be found.

\begin{definition} Two trihexes are \emph{equivalent} if they are not only isomorphic as graphs but if there is also an orientation-preserving homeomorphism of the plane that takes one graph to the other. 
\end{definition}

Deza and Dutour (\cite{deza2005zigzag} and \cite{deza2008geometry}) describe a family of 2-connected trihexes denoted by $G_n$ or $T_n$, where $n$ is half the number of hexagons. We will refer to these trihexes as godseyes, after the woven yarn craft figure that they resemble.

\begin{definition}
A \emph{godseye} is a trihex that consists of two adjacent triangles, surrounded by one or more nested pairs of hexagons, with two more adjacent triangles on the outside. The hexagons in each nested pair meet along opposide sides. See Figure \ref{fig:godseye}. 
\end{definition}

\begin{figure}
    \centering
    \includegraphics[width=4cm, height=4cm]{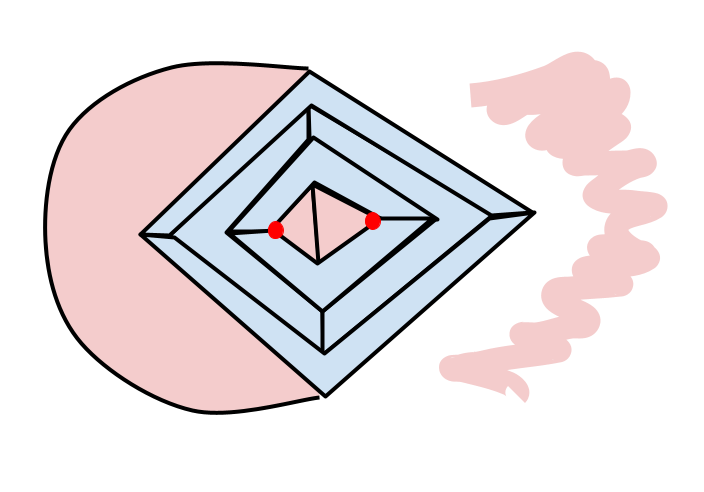}
    \caption{Godseye with three pairs of hexagons. }
    \label{fig:godseye}
\end{figure}

Godseyes can have any even number of hexagons; however, a standard Euler characteristic argument shows that every trihex has exactly four faces. See, for example, \cite{grunbaum_motzkin_1963}. 
The argument is as follows.
Let $f_6$ be the number of hexagons in the trihex and $f_3$ be the number of triangular faces. 
The number of faces is $F = f_6 + f_3$. The number  of edges is $E = \dfrac{6f_6 + 3f_3}{2}$, since each hexagonal face has six edges, each triangular face has three edges, and each edge is shared by two faces. The number of vertices is $V = \dfrac{6f_6 + 3f_3}{3}$, since each hexagonal face has six vertices, each triangular face has three vertices, and each vertex is shared by three faces. By Euler's formula, we have $V - E + F = 2$. Therefore, 
    $\dfrac{6f_6+3f_3}{3}-\dfrac{6f_6+3f_3}{2}+f_6+f_3=2 $, 
 which implies $f_3=4$.

Euler's formula  places no restrictions on the number of hexagonal faces; however, Gr\"unbaum and Motzkin showed that only even numbers of hexagonal faces can be achieved \cite{grunbaum_motzkin_1963}.

\section{Building trihexes from spines and belts}

\label{sec:examples}

In this section, we describe ways to construct trihexes out of  strings of hexagons capped by triangles (``spines''), possibly with rings of hexagons (``belts'') separating the spines. Our construction echoes the construction given by Gr\"unbaum and Motzkin in \cite{grunbaum_motzkin_1963} but adds the consideration of ``offset'' defined below.

\begin{definition}
    A \emph{belt} is a circuit of distinct hexagonal faces in a trihex such that
each hexagon is adjacent to its neighbors on opposite edges  \cite{deza2005zigzag}. 
\end{definition}

\begin{definition}
\label{def:spine}
A \emph{spine} is a collection of distinct faces in a trihex $F_0, F_1, \cdots , F_{s+1}$, with $s \geq 0$, such that
\begin{enumerate}[(i)]
    \item $F_0$ and $F_{s+1}$ are triangles, 
    \item $F_1, F_2, \cdots , F_s$ are hexagons, and
    \item For each hexagon $F_i$, $1 \leq i \leq s$, $F_i$ is adjacent to $F_{i-1}$ and to $F_{i+1}$ along opposite edges of $F_i$.
\end{enumerate}

\begin{figure}
    \centering
    \includegraphics[ height=6cm]{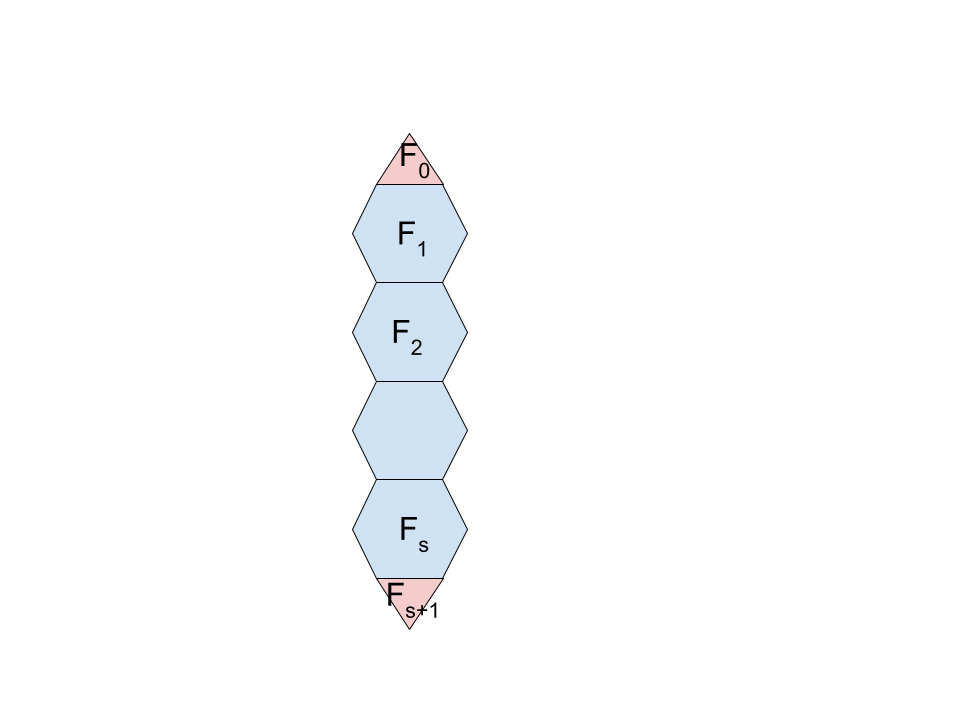}
    \caption{Spine of length 4.}
    \label{fig:spine}
\end{figure}

The \emph{internal edges} of the spine are the edges shared by $F_i$ and $F_{i+1}$ for $0 \leq i \leq s$  and the \emph{external edges} of the spine are all the other edges. The \emph{length} of the spine is the number $s$ of hexagonal faces between the triangular faces. Note that a spine of length 0 is a pair of triangles that share an edge. We refer to the triangle $F_0$  as the \emph{head triangle} of the spine and the triangle $F_{s+1}$ as the \emph{tail triangle}. Note that  which triangle is considered the head triangle and which is considered the tail triangle depends only on the choice of numbering. The \emph{head vertex}  of the spine is the ``tip'' vertex of the head triangle, that is, the vertex that is not on an internal edge. The \emph{tail vertex} of the spine is the ``tip'' vertex of the tail triangle.
\end{definition}

A trihex can be created from two spines of length $s$ by attaching them along the $4s + 4$ external edges in each of their boundaries. This can be done in multiple ways. See Figure~\ref{fig:attachingSpines} for examples with $s = 5$.

\begin{figure}
    \begin{subfigure}{.3\textwidth}
        \centering
    \includegraphics[height=5cm]{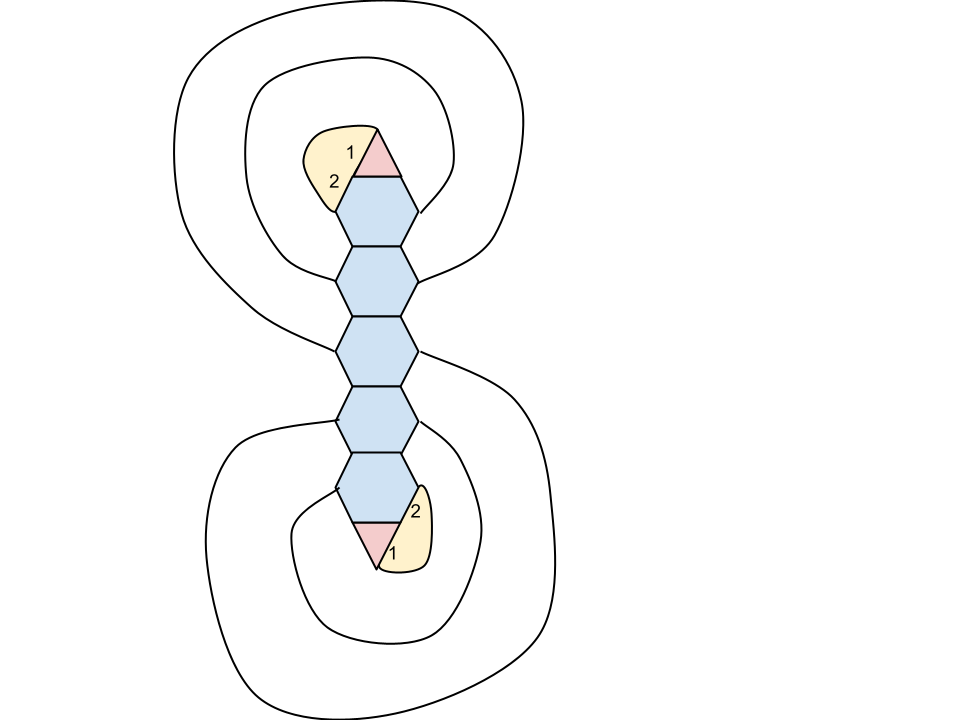}
    \caption{Offset 0}
    \label{fig:offset0_ring0}
    \end{subfigure}
    \begin{subfigure}{.3\textwidth}
        \centering
    \includegraphics[height=5cm]{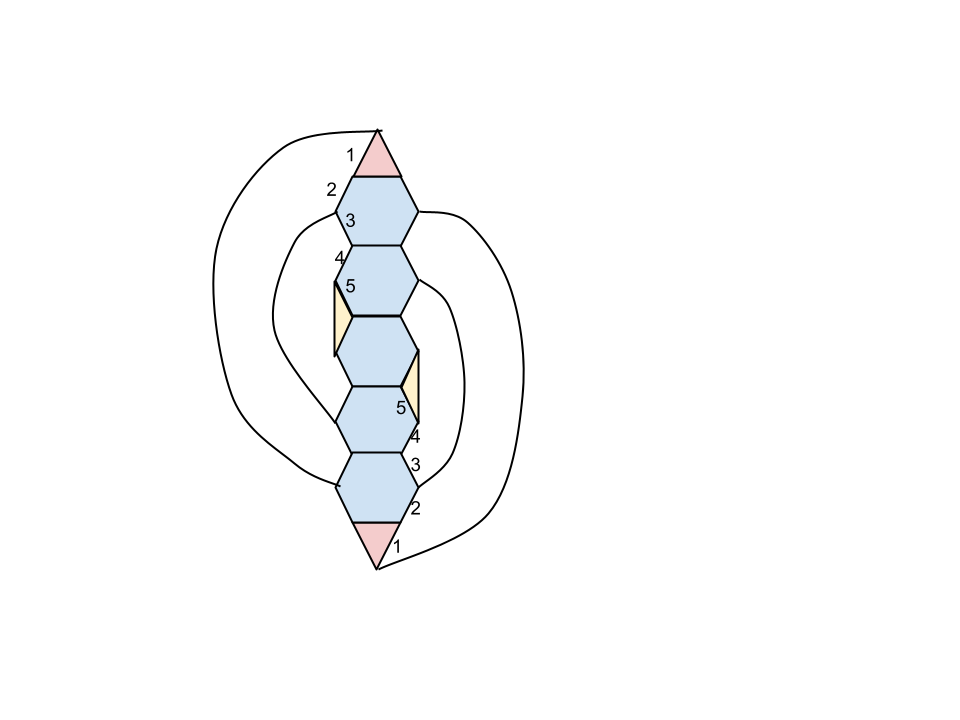}
    \caption{Offset 2}
    \label{fig:offset2_ring0}
    \end{subfigure}
    \begin{subfigure}{.3\textwidth}
        \centering
    \includegraphics[height=5cm]{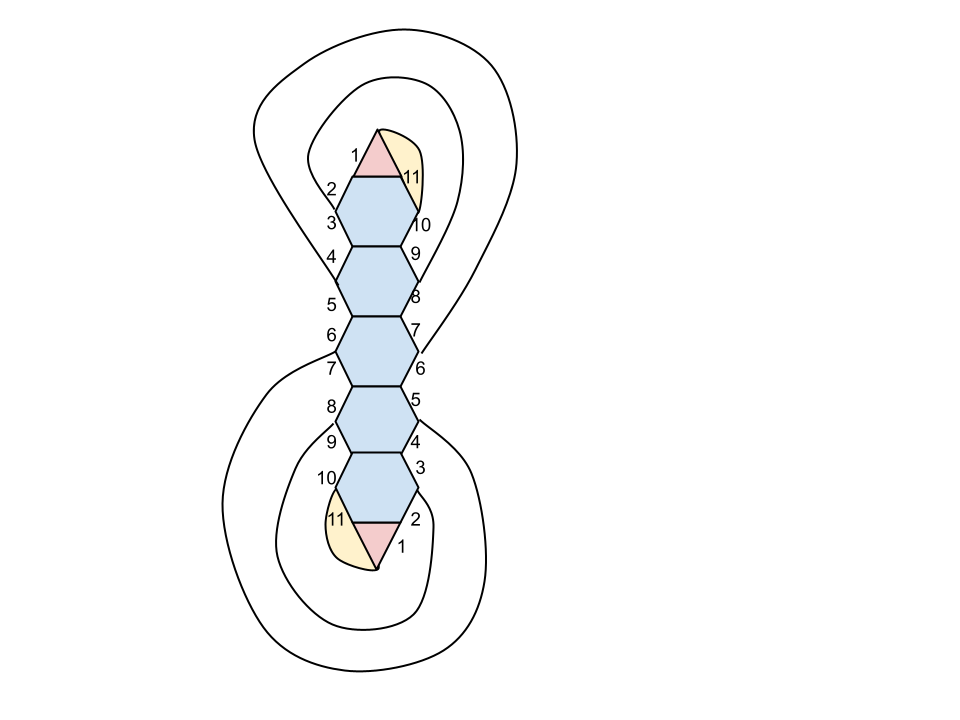}
    \caption{Offset 5}
    \label{fig:offset5_ring0}
    \end{subfigure}
\caption{Attaching spines}
\label{fig:attachingSpines}
\end{figure}

\begin{definition}
Suppose that two spines of length $s$ are identified along their $4s+4$ external  edges. Starting with the head or tail vertex of one spine, travel counterclockwise around the boundary edges of this spine, until either a head or a tail vertex of the other spine is encountered, and count the number of edges traversed.  We say that the two spines are attached with \emph{offset} $i \mod (s+1)$ if the number of edges traversed is $2i + 1$. 
\end{definition}

To see that offset is well-defined, first note that the head vertex (or tail vertex) of the second spine must be identified to a vertex of the first spine where two faces of the first spine already meet. Otherwise, the trihex would not be 3-regular. Therefore, the number of edges traversed between the head or tail vertex of the first spine and a head or tail vertex of the second spine must be an odd number, which has the form $2i+1$ for some number $i$. In addition, since each spine has head and tail vertices that are $2s+2$ edges apart, the number of edges traversed, going counterclockwise, to get from the \emph{head} vertex of the first spine to any head or tail vertex of the second spine will be the same number $\mod (2s+2)$. It will also be the same number $\mod (2s+2)$ as the number of edges traversed to get from the \emph{tail} vertex of the first spine to any head or tail vertex of the second spine. Since $2i + 1 \equiv 2j + 1 \mod(2s+2)$, if and only if $i \equiv j \mod (s+1)$, the  offset is well-defined $\mod(s+1)$, no matter which head and tail vertices are used. It makes no difference which spine is considered the first spine and which is considered the second, since the same paths of edges are traversed whether traveling counterclockwise around one spine or the other, when going between a head or tail vertex of one spine and a head or tail vertex of the other.

For integers $s \geq 0$ and $b > 0$, we can also build a trihex out of a pair of spines of length $s$ together with $b$ belts of $2s+2$ hexagons, where the belts  lie in between the two spines and encircle each spine. See Figure~\ref{fig:trihexExamples1}. 
As before, there are a variety of ways to attach the second spine onto the outermost belt, depending on where the head triangle of the second spine is inserted. Again, these different insertion points can be characterized by offsets.

\begin{figure} 
    \begin{subfigure}{.5\textwidth}
        \centering
    \includegraphics[height = 5 cm]{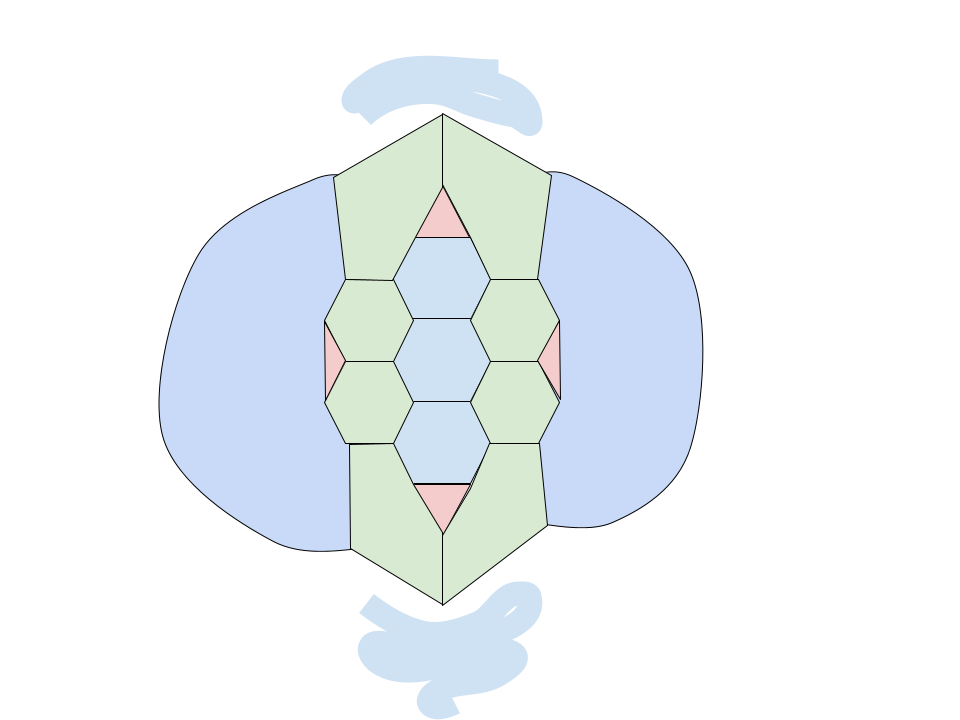}
    \caption{Two spines of length 3 with 1 belt and offset 1.}
    \end{subfigure}
        \begin{subfigure}{.5\textwidth}
        \centering
    \includegraphics[height = 5 cm]{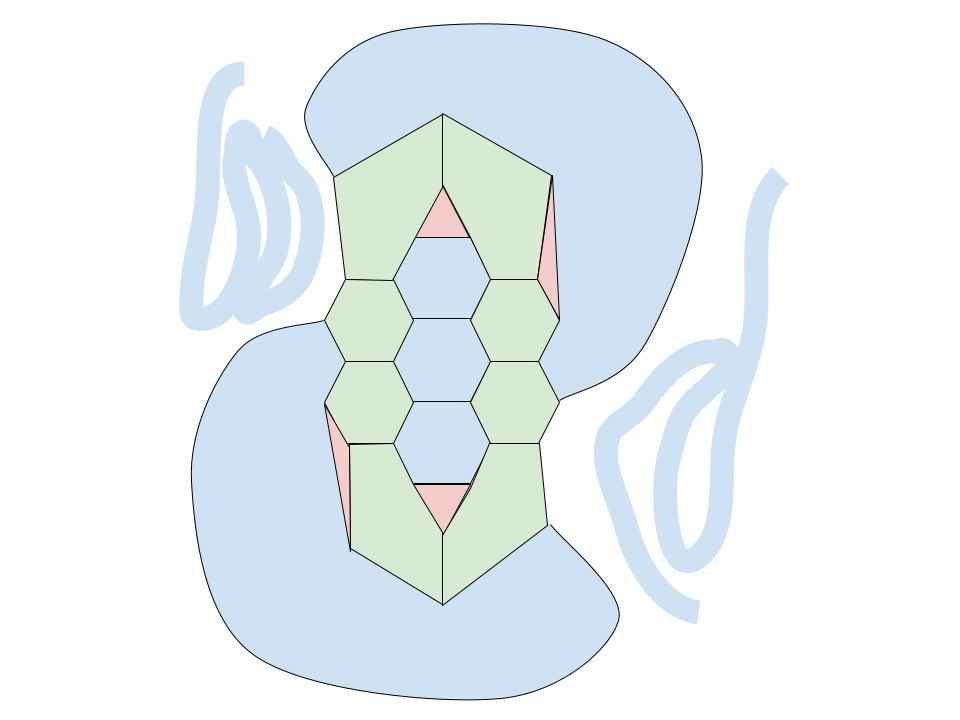}
    \caption{Two spines of length 3 with 1 belt and offset 2.}
   \end{subfigure}
    
    \begin{subfigure}{.6\textwidth}
        \centering
    \includegraphics[height = 5 cm]{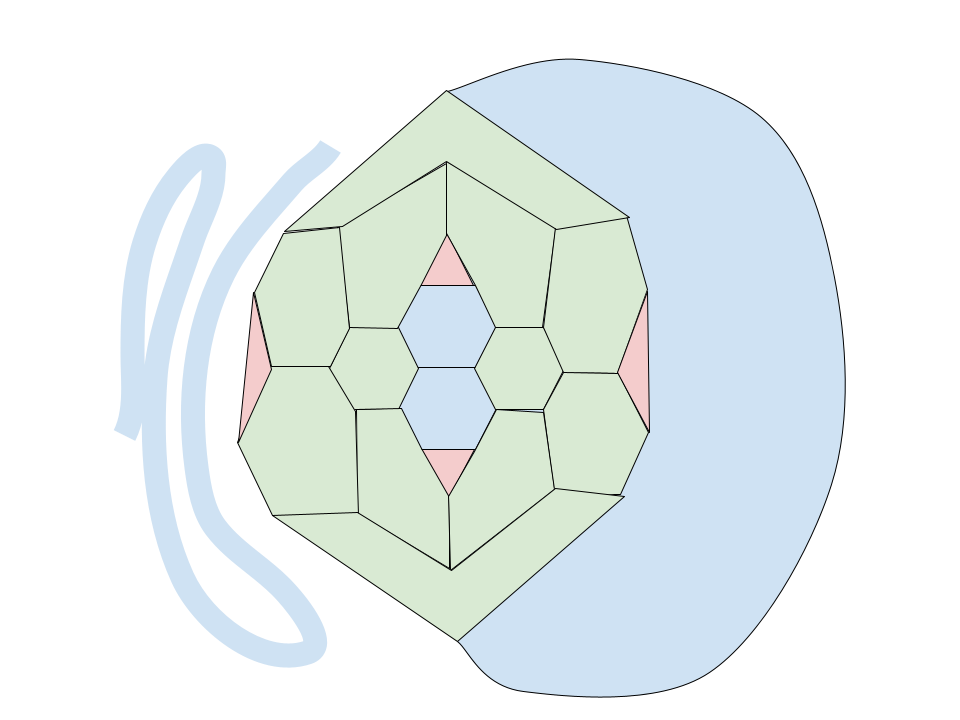}
    \caption{Two spines of length 2 with 2 belts and offset 0.}
    \end{subfigure}
\caption{Spines with belts between them.}
\label{fig:trihexExamples1}
\end{figure}

Suppose first that we have only one belt. Suppose we delete the belt and slide the two spines towards each other to fill in the space. If we slide them straight towards each other, along the edges of the belt that previously separated them, and then shift each spine slightly,  either clockwise or counterclockwise around the other spine, then we form a new trihex with no belts between the spines. See Figure~\ref{fig:shiftingToFindOffset1}. Using the convention that we always shift clockwise, we can define the offset of the original trihex to be the offset of the new trihex with no belts between the spines (after shifting clockwise).

\begin{figure}
    \begin{subfigure}{.5\textwidth}
    \centering
    \includegraphics[height = 5 cm]{trihex3-1-2.png}
    \caption{The original trihex with one belt.}
    \label{fig:trihex3-1-2Before}
    \end{subfigure}
        \begin{subfigure}{.5\textwidth}
        \centering
    \includegraphics[height = 5 cm]{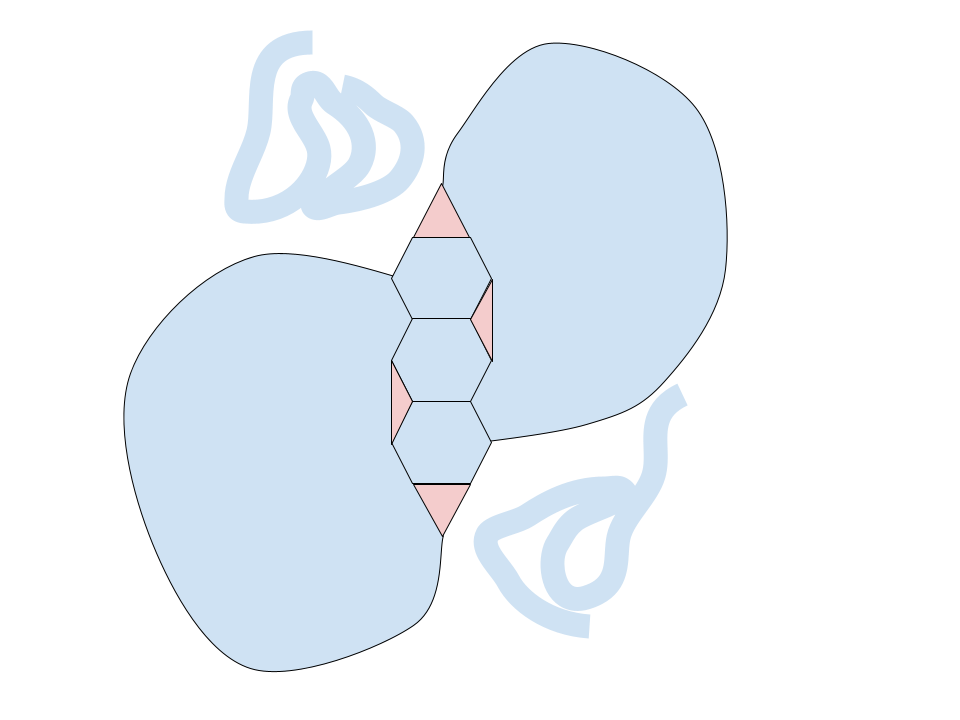}
    \caption{After deleting the belt and shifting clockwise, the convention that we  use. The offset of this trihex is 2, and therefore the offset of the original trihex is also 2.}
    \label{fig:trihex3-1-2AfterClockwiseShift}
    \end{subfigure}
    \begin{subfigure}{.5\textwidth}
        \centering
    \includegraphics[height = 5 cm]{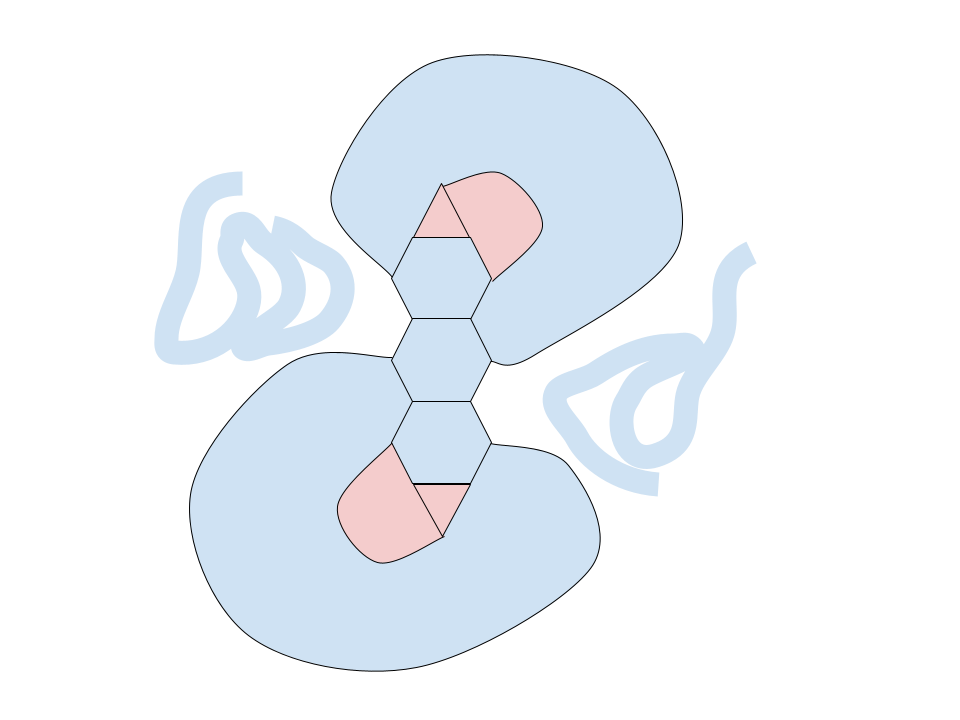}
    \caption{After deleting the belt and shifting counterclockwise.}
    \label{fig:trihex3-1-2AfterCounterclockwiseShift}
    \end{subfigure}

\caption{Shifting clockwise vs. counterclockwise.} \label{fig:shiftingToFindOffset1}
\end{figure}

To find the offset when there are additional belts between the spines, we repeat the process of deleting belts and shifting the remaining pieces, starting from the outermost belt and working in.

\begin{definition}
When there are one or more belts between the spines, successively delete the belts, starting from the outermost belt, each time shifting one spine clockwise around the other spine. The offset of the original trihex is defined to be the offset of  the resulting trihex that has no belts between the spines. 
\end{definition}

Note that shifting one spine clockwise around the other spine gives the same configuration as shifting the other spine clockwise around the first spine. Therefore, offset is well-defined irrespective of which spine is shifted with respect to the other and which belt is considered outermost vs. innermost.

For example, the original trihex in Figure~\ref{fig:shiftingToFindOffset2} has offset 0.

\begin{figure}
    \begin{subfigure}{.5\textwidth}
        \centering
    \includegraphics[height = 5 cm]{trihex2-2-0.png}
    \caption{The original trihex with two belts.}
    \end{subfigure}
    \begin{subfigure}{.5\textwidth}
        \centering
    \includegraphics[height = 5 cm]{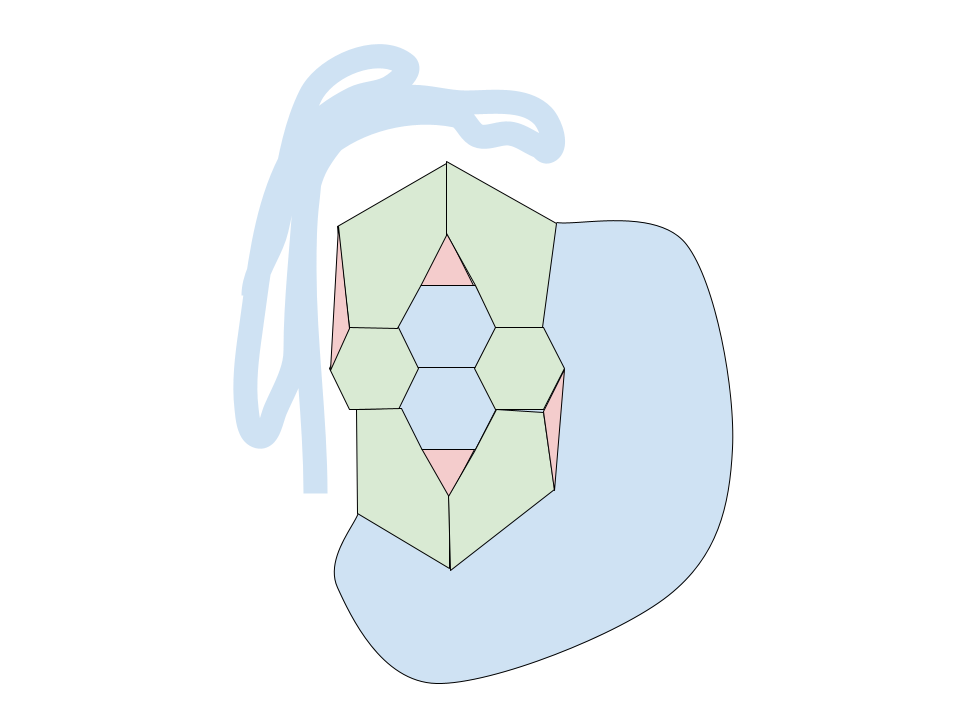}
    \caption{After deleting the one belt and shifting clockwise.}
    \end{subfigure}
    \begin{subfigure}{.5\textwidth}
        \centering
    \includegraphics[height = 3 in, width = 3.7 in]{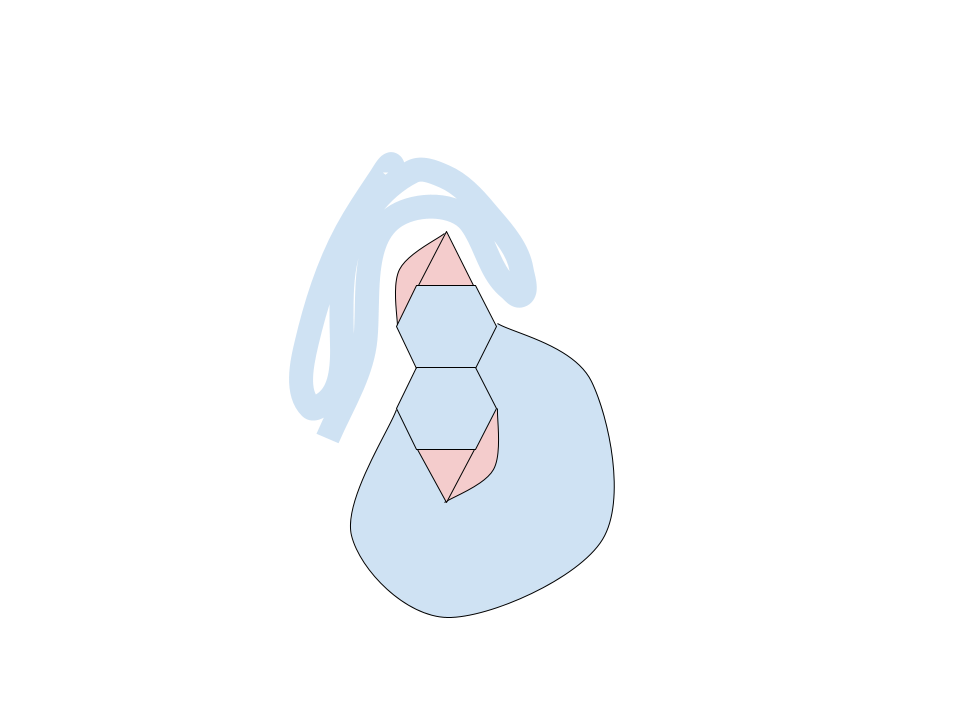}
    \caption{After deleting the the second belt and shifting clockwise again. The offset of this trihex is 0, and so the offset of the original trihex is 0.}
    \end{subfigure}
\caption{Shifting to find offset.}
\label{fig:shiftingToFindOffset2}
\end{figure}

\begin{definition} Let $s \geq 0$, $b \geq 0$, and $0 \leq f \leq s$. If a
 trihex can be formed from two spines of length $s$ with $b$ belts between them and with offset $f$, then the \emph{signature} of the trihex is the ordered triple $(s, b, f)$.
\end{definition}

It is clear from the contruction that any two trihexes built from two spines of length $s$, with $b$ belts between them and offset $f$ are equivalent. In addition, Gr\"unbaum and Motzkin \cite{grunbaum_motzkin_1963} show that any trihex can be decomposed into spines and surrounding belts. Therefore, any trihex can be described with a signature $(s, b, f)$ for some $s \geq 0$, $b \geq 0$, and $0 \leq f \leq s$.

We summarize these facts in the following: 

\begin{theorem}
    \begin{enumerate}[(i)]
    \item Given $s \geq 0$, $b \geq 0$, $0 \leq f \leq s$, there exists a trihex with signature $(s, b, f)$.
    \item Any two trihexes with the same signature are equivalent.
    \item Any trihex can be described with a signature $(s, b, f)$ for some $s \geq 0$, $b \geq 0$, and $0 \leq f \leq s$.
    \end{enumerate}
    \label{thm:achieveAllSigs}
\end{theorem}

The signature for a trihex is not unique: decomposing a trihex in different ways can produce three different signatures, as detailed in Section~\ref{sec:equivalent signatures}.

\section{Trihexes and hexagonal grid coverings}
\label{sec:grid}

In this section, we create a hexagonal tiling of the plane that covers a given trihex. Doing so allows us to develop rules for finding alternative signatures for a trihex in Section~\ref{sec:equivalent signatures}. 

Consider a hexagonal tiling of the plane, made up of regular hexagons arranged in vertical strips such that two sides of each hexagon are horizontal. Superimpose a grid of parallelograms on the hexagonal tiling, such that each vertex of each parallelogram lies in the center of a hexagon, and each parallelogram has two vertical sides. See Figure~\ref{fig:fundamentalDomain}.

\begin{figure}
    \centering
    \includegraphics[height=8cm]{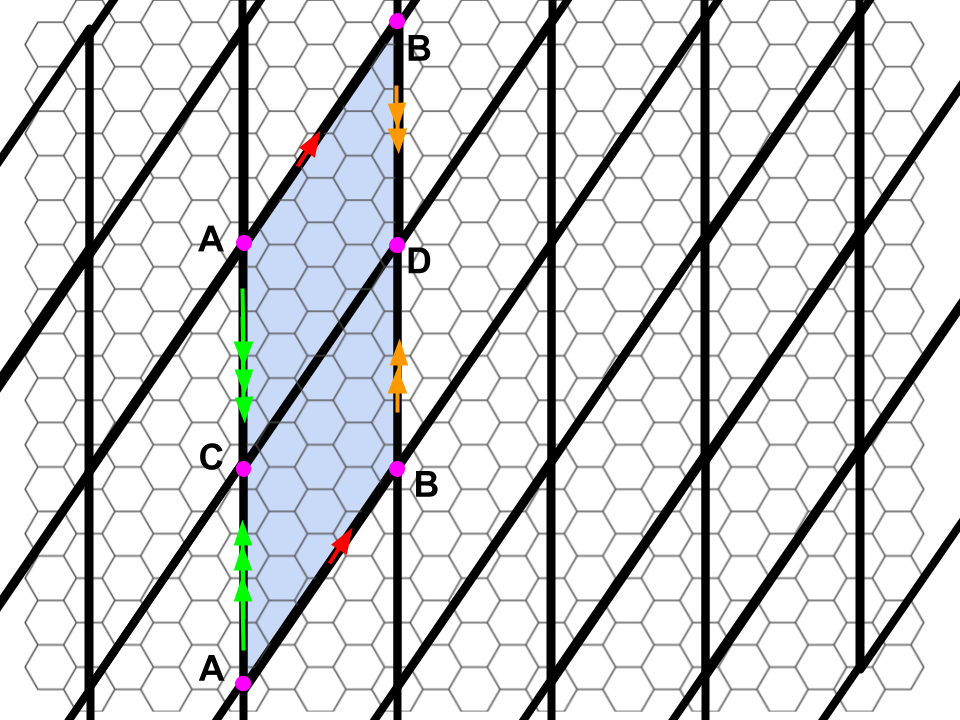}
    \caption{Hexagonal grid with fundamental domain.}
    \label{fig:fundamentalDomain}
\end{figure}

Consider the group of transformations of the plane generated by $180^\circ$ rotations around each vertex of the parallelogram grid. Form a quotient space (an orbifold) by identifying all points in the same orbit of this transformation group. A fundamental domain for this transformation group can be given by a pair of adjacent parallelograms, like the two parallelograms shaded blue in Figure~\ref{fig:fundamentalDomain}, as explained below.

The orbits of points in this pair of parallelograms cover the entire plane, since a rotation around point C and then around the upper point marked A translates the pair of parallelograms up, and a rotation around C and then around the lower point marked A translates the parallelograms down. Repeating these pairs of rotations translates the parallelograms over an entire vertical strip. A rotation around C moves and inverts this vertical strip to cover the strip to the left, while a rotation around D covers the strip to the right. Repeated alternating rotations around points C and D covers all additional vertical strips to the left and the right. 

No smaller subset of this double parallelogram region has orbits that cover the entire plane, by the following reasoning. A product of two $180^\circ$ rotations is a translation through a vector twice the length of the vector connecting the rotocenters. A product of translations through two vectors is a translation through the sum of the vectors. Therefore, a product of an even number of $180^\circ$ rotations around parallelogram grid vertices is a translation through a vector that is some sum $2m\vec{AB} + 2n\vec{AC}$ for some integers $m$ and $n$. A product of an $180^\circ$ rotation and a translation is a $180^\circ$ rotation whose rotocenter is the original rotocenter shifted by half the translation vector. Therefore, a product of an odd number of $180^\circ$ rotations around parallelogram vertices is a $180^\circ$ rotation around a rotocenter that is a parallelogram grid vertex shifted by $\dfrac{1}{2} (2m\vec{AB} + 2n\vec{AC})$ for some integers $m$ and $n$, which is just another parallelogram grid vertex. The points interior to the double parallelogram region cannot be transformed onto each other by $180^\circ$ rotations around parallelogram grid vertices or translations by linear combinations of $2\vec{AB}$ and $2\vec{AC}$. Therefore, the double parallelogram region is a minimal size region whose orbits cover the plane, i.e. a fundamental domain. 

Although no points in the interior of the fundamental domain are identified under the transformation group, many pairs of points on the edges of the fundamental domain are identified with each other. This is indicated by the arrows in Figure~\ref{fig:fundamentalDomain}: a rotation around point $D$ identifies the edges above and below $D$, a rotation around $C$ identifies the edges above and below $C$, and rotation around $A$ followed by rotation around $C$ identifies the top and bottom edges between the points marked $A$ and $B$. 
After identifying edges, the resulting quotient is a topological sphere. See Figure~\ref{fig:orbifoldSphere}.

\begin{figure}
    \centering
    \includegraphics[height=8cm]{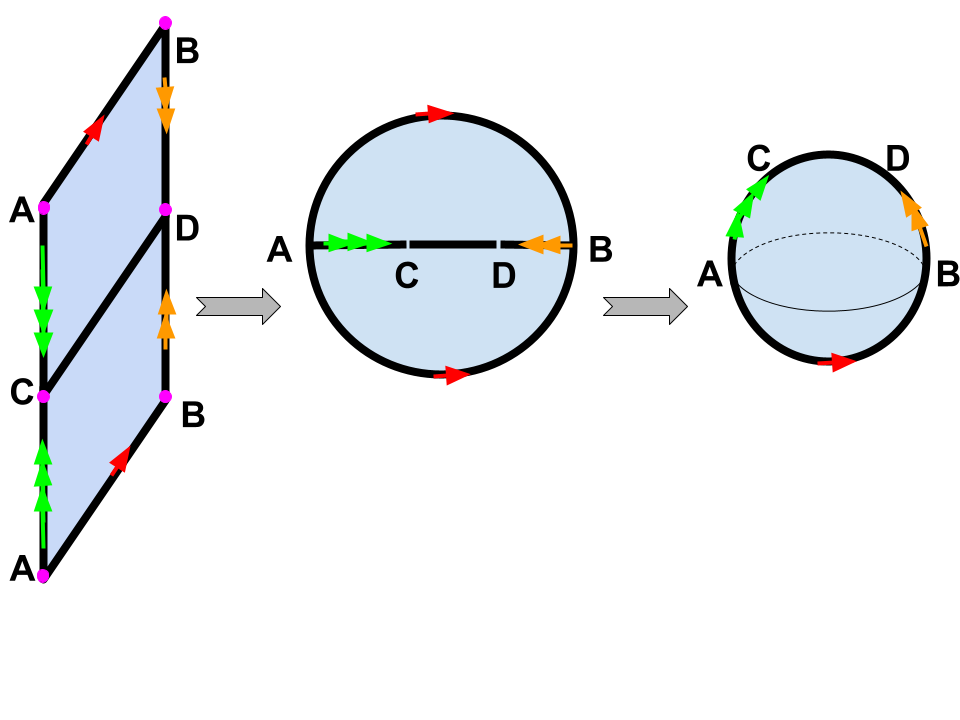}
    \caption{Identifying edges creates a topological sphere.}
    \label{fig:orbifoldSphere}
\end{figure}

Note that the hexagonal tiling is preserved by all of the $180^\circ$ rotations. Therefore, it can be projected via the quotient map down to the quotient sphere. Each hexagon that lies entirely inside the fundamental domain will project onto a hexagon on the sphere. Also, the partial hexagons that are cut off by edges of the fundamental domain, but do not contain the vertices marked $A$, $B$, $C$, and $D$, attach up in pairs and therefore also project onto hexagons in the quotient sphere. The partial hexagons that contain the vertices marked $A$, $B$,  $C$, and $D$ get identified in such a way that only half a hexagon appears in the quotient sphere. This half hexagon is has its diameter identified by a $180^\circ$ rotation, so that it forms a triangle in the quotient sphere. See Figure~\ref{fig:hexagonWrapsToTriangle}. Therefore, the hexagonal tiling naturally forms a pattern of hexagons and triangles on the quotient sphere, with three faces around each vertex. So all the requirements for a trihex are satisfied. 

\begin{figure}
    \centering
    \includegraphics[height=4cm]{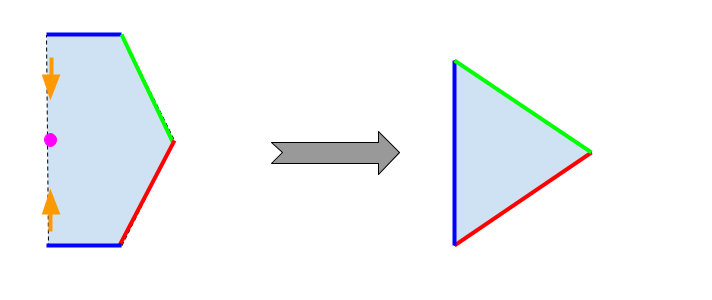}
    \caption{The quotient of a hexagon under a $180^\circ$ rotation around its center point.}
    \label{fig:hexagonWrapsToTriangle}
\end{figure}

A signature for a trihex described as the quotient of a hexagonal tiling can be read off directly from the tiling, as illustrated in  Figures~\ref{fig:hexagonalGridForOffset} and \ref{fig:trihex4-3-3Cover}. Each vertical string of hexagons between vertices labeled $A$ and $C$, together with the hexagons centered at $A$ and $C$, projects to one spine, and each vertical string of hexagons between vertices labeled $B$ and $D$, together with the hexagons centered at $B$ and $D$, projects to a second spine. The  hexagons that lie in vertical strips between the vertical sides of a fundamental domain project to belts between the two spines. Therefore the number $s$ in the signature $(s, b, f)$ is the number of hexagons that lie in a vertical strip strictly between the hexagons at vertex $A$ and vertex $B$.  The number $b$ is the number of vertical columns of hexagons in the tiling that lie entirely between the two vertical edges of the fundamental domain. 

To find the offset for the trihex, choose a hexagon centered at a vertex $A$ and translate it along the diagonal strip of hexagons in the approximately southwest (SW) to northeast (NE) direction, until it coincides with a hexagon in the vertical column of hexagons containing vertices $B$ and $D$.  If we hit a hexagon that is $k$ hexagons below a vertex $B$ or $D$, then our trihex will have offset $k$. This is because a hexagon at vertex $A$ projects to a head triangle in the trihex, and translating this hexagon one column to the right in the SW to NE direction corresponds to deleting one belt and shifting the head triangle clockwise in the trihex. The ultimate position of the hexagon at vertex $A$ after translating all the way to the right edge of the fundamental domain corresponds to the ultimate position that the head vertex of the head triangle is inserted along the second spine in the trihex, which is the offset. For example, the fundamental domain shown in Figure~\ref{fig:trihex4-3-3Cover} covers the trihex with signature $(4, 3, 3)$.

\begin{figure}
    \centering
    \includegraphics[height=8cm]{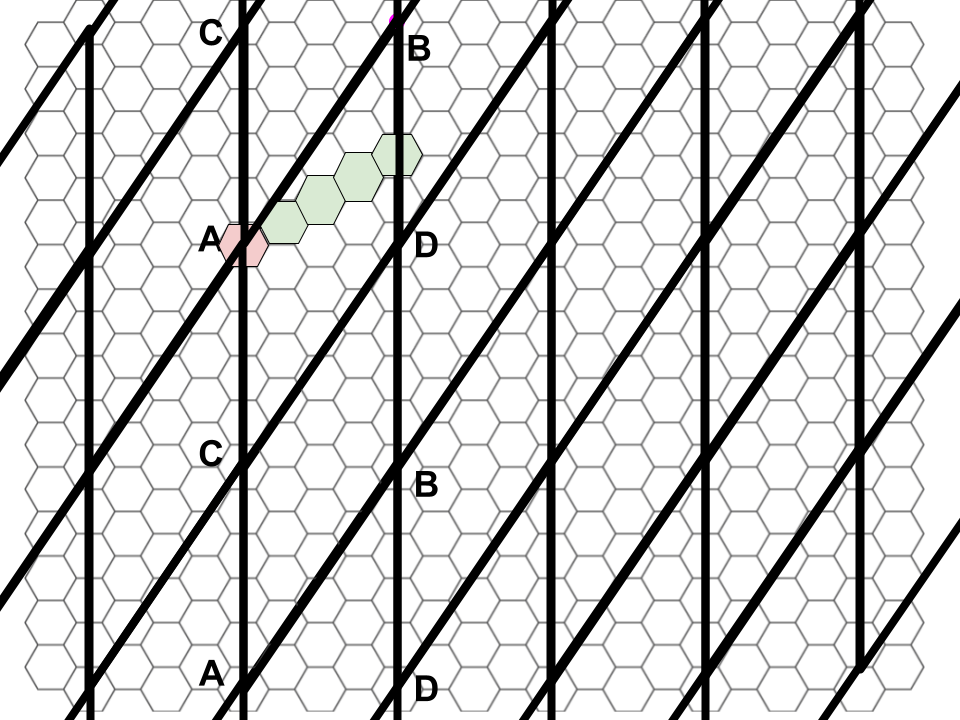}
    \caption{Calculating offset.}
    \label{fig:hexagonalGridForOffset}
\end{figure}

\begin{figure}
           \begin{subfigure}{.4\textwidth}
    \centering
    \includegraphics[height=8cm]{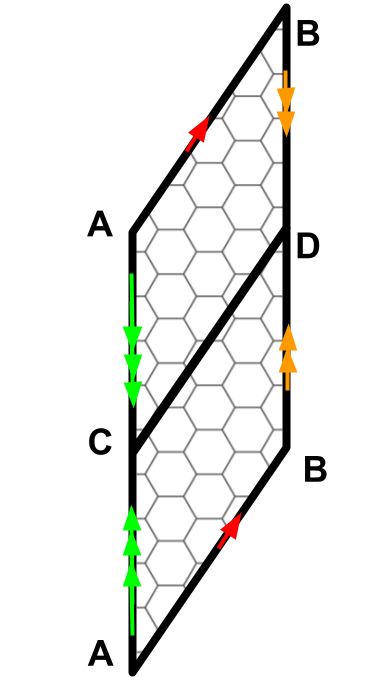}
    \end{subfigure}
           \begin{subfigure}{.4\textwidth}
    \centering
    \includegraphics[height=7cm]{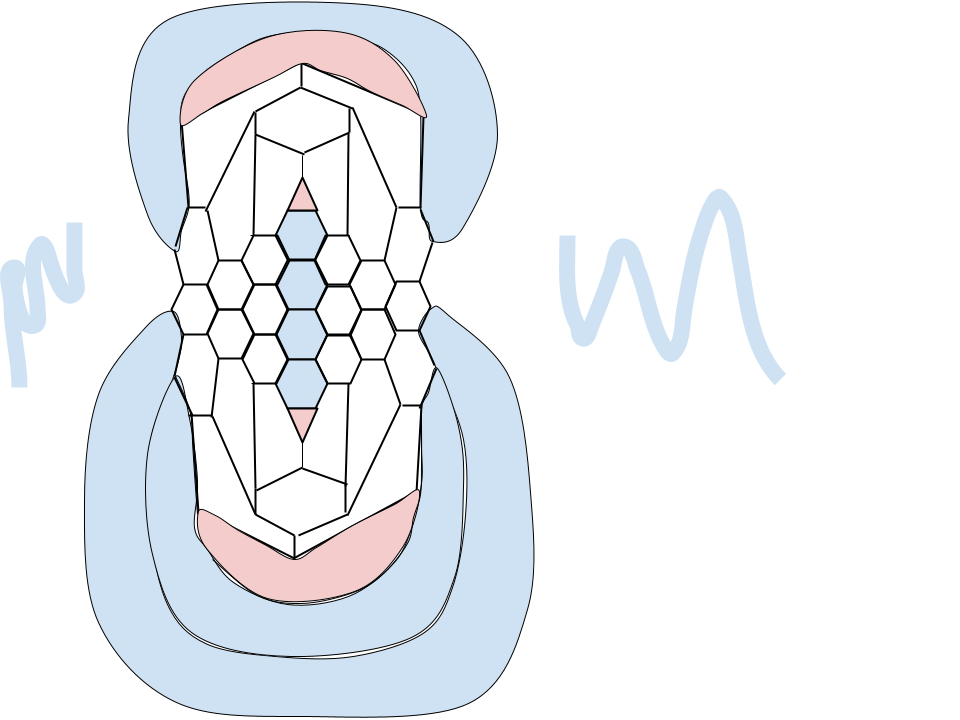}
    \end{subfigure}

    \caption{The fundamental domain that covers the trihex with signature $(4, 3, 3)$.}
    \label{fig:trihex4-3-3Cover}
\end{figure}

It is now possible to conclude the following:

\begin{theorem}
\label{thm:allTrihexesComeFromTilings}
Every trihex can be produced as  the quotient of a hexagonal tiling of the plane under a group of transformations generated by $180^\circ$ rotations around the vertices of a superimposed parallelogram grid.
\end{theorem}

\begin{proof} Recall from Theorem~\ref{thm:achieveAllSigs} that any trihex can be described with a signature $(s, b, f)$ with $s \geq 0$, $b \geq 0$, and $0 \leq f \leq s$. For any such triple of integers $(s, b, f)$, build a hexagonal tiling with a superimposed parallelogram grid  as follows. Start with a tiling  of the plane by regular hexagons in which two sides of each hexagon are horizontal.  Put two vertices of a parallelogram on the centers of two hexagons in the same vertical column that are  separated by $s$ hexagons strictly between them. Put the other two vertices of the parallologram on the centers of hexagons in another vertical column, $b+1$ columns to the right of the first column. This second pair of vertices should also be separated by $s$ hexagons strictly between them. Shift the second pair of vertices up or down as needed, so that when the hexagons containing the first pair of vertices are translated along a SW to NE diagonal, through $b+1$ columns of hexagons, they end up $f$ hexagons below the hexagons occupied by the second pair of vertices. We now have one parallelogram whose vertices lie on the centers of hexagons. Tile the plane with translated copies of this parallelogram to create a  parallelogram grid. The quotient of the hexagonal tiling by the group generated by $180^\circ$ rotations around parallelogram vertices  is a trihex with signature $(s, b, f)$.

\end{proof}

The process of creating a hexagonal tiling that covers a given trihex can be thought of as ``unwrapping'' the trihex around each triangle. Figure~\ref{fig:4-1-2} shows the unwrapping of the  trihex $(4, 1, 2)$. Numbered hexagons in the trihex correspond to numbered hexagons in the hexagonal tiling.

\begin{figure}
              \begin{subfigure}{.4\textwidth}
    \centering
    \includegraphics[height=8cm]{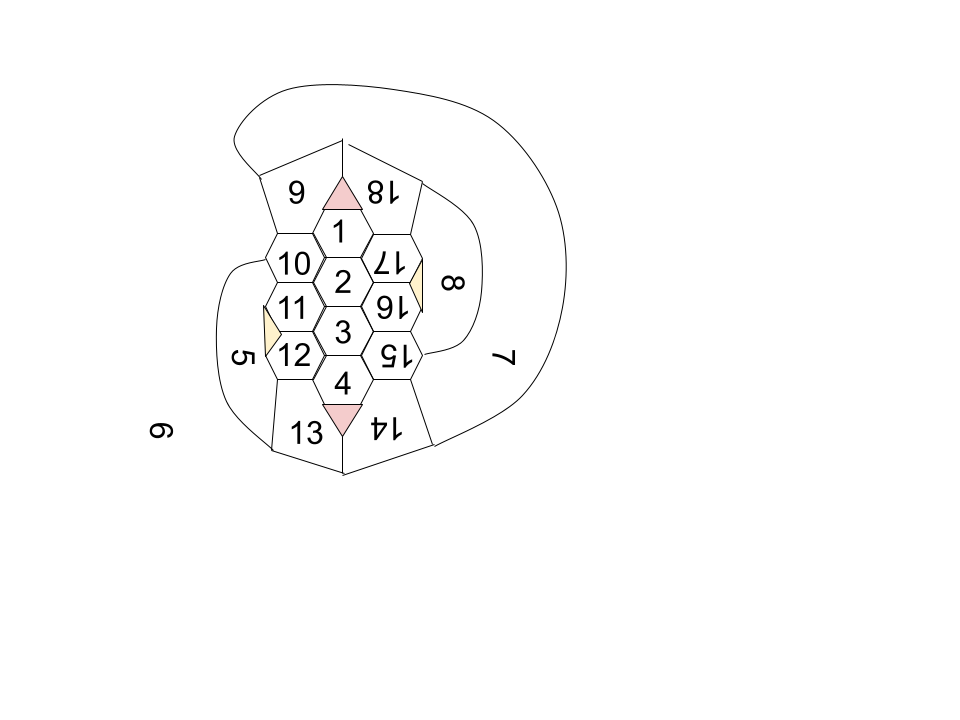}
    \label{fig:4-1-2Trihex}
    \end{subfigure}
            \begin{subfigure}{.4\textwidth}
    \centering
    \includegraphics[height=8cm]{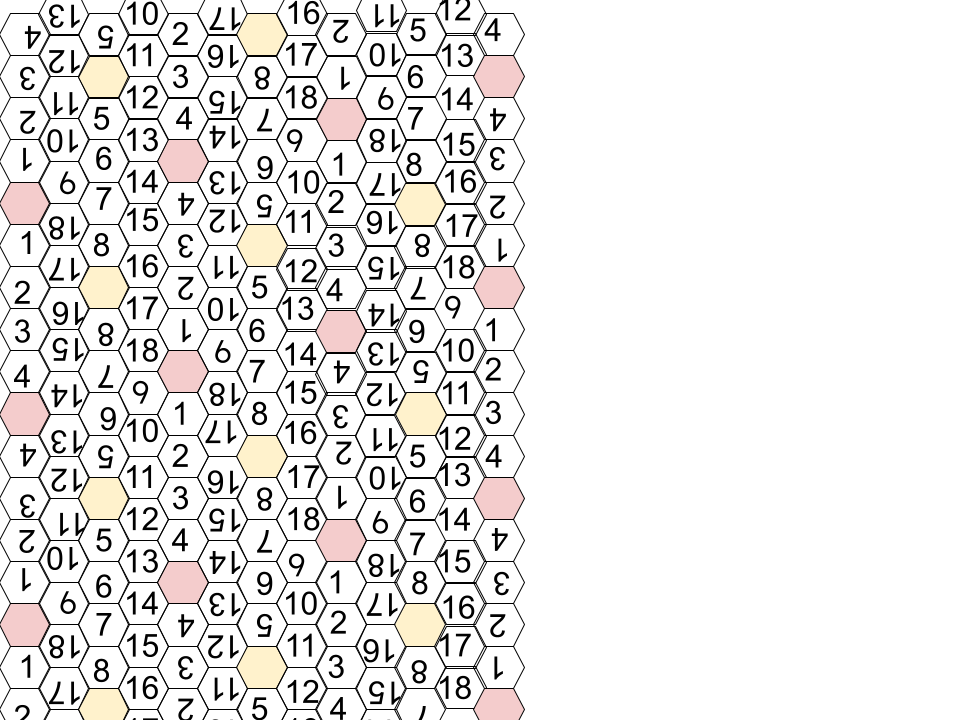}
    \label{fig:4-1-2TrihexTiling}
    \end{subfigure}

    \caption{The trihex with signature $(4, 1, 2)$ and the hexagonal tililng that covers it.}
    \label{fig:4-1-2}
\end{figure}

\section{Equivalent signatures}
\label{sec:equivalent signatures}

The signature $(s, b, f)$ for a trihex is not necessarily unique. This section develops rules for finding alternative signatures for a trihex based on a given signature and shows that there is a one to one correspondence between equivalence classes of signatures, as defined in this section, and equivalence classes of trihexes, as defined in Section~\ref{sec:defs}.

Start with a trihex with signature $(s_1, b_1, f_1)$. 
As described in Theorem~\ref{thm:allTrihexesComeFromTilings}, this trihex is the quotient of a hexagonal tiling of the plane under a group of transformations generated by $180^\circ$ rotations around vertices of a superimposed parallelogram grid.  We will refer to the hexagons centered at vertices of the parallelogram grid as  ``special hexagons''. These hexagons cover triangles in the trihex. The quotients of vertical columns of hexagons that contains special hexagons are spines of the trihex. We also use the phrase ``vertical spine'' to refer to a strips of vertical hexagons between pairs of special hexagons.
For example, a hexagonal tiling for trihex $(5, 2, 2)$ is shown in Figure~\ref{fig:5-2-2Tiling}. Vertical spines are shaded blue. Special hexagons are shaded pink and yellow. The superimposed parallelogram grid is not drawn.

\begin{figure}
    \centering
    \includegraphics[height=8cm]{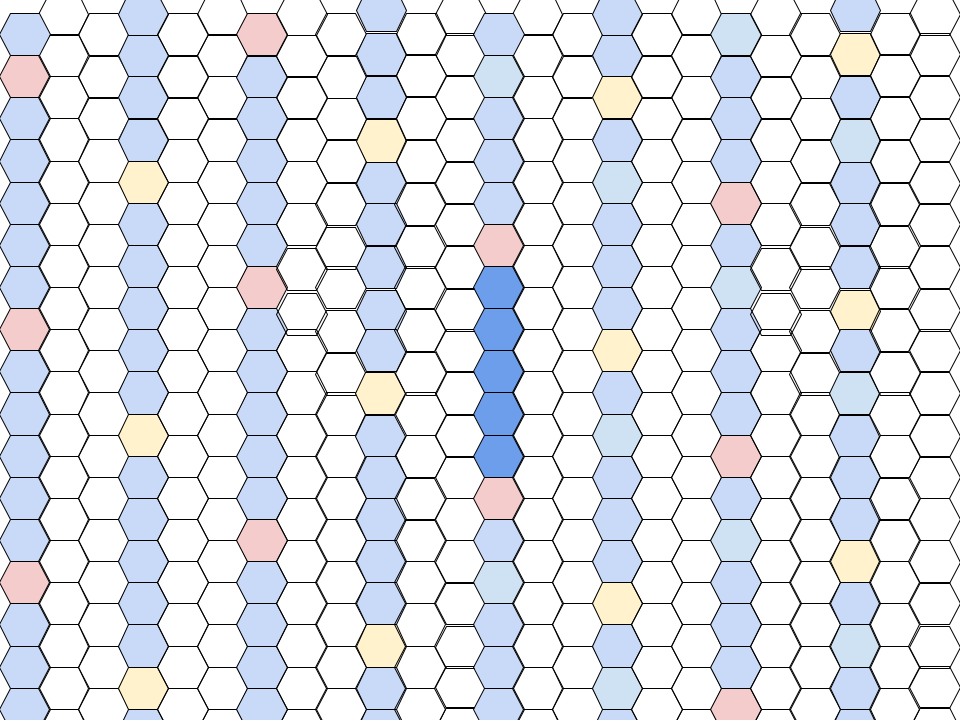}
    \caption{A hexagonal tiling for the trihex $(5, 2, 2)$.}
    \label{fig:5-2-2Tiling}
\end{figure}

There are two additional ways to describe the trihex. Instead of using vertical strings of hexagons to make spines, we could build spines by setting off from a special hexagon at an angle  $60^\circ$ clockwise from due north or at an angle $120^\circ$ clockwise from due north. We will refer to these directions as the southwest (SW) to northeast (NE) direction and the northwest (NW) to southeast (SE) direction.  See Figures~\ref{fig:SWtoNESpine} and~\ref{fig:NWtoSESpine}. We will use the notation $(s_2, b_2, f_2)$ to refer to the signature when we build spines in the SW to NE direction and $(s_3, b_3, f_3)$ to refer to the signature when we go in the NW to SE direction.

Suppose we go from SW to NE. See Figure~\ref{fig:SWtoNESpine}, where four special hexagons are labelled $T, B, L, R$ (top, bottom, left, and right). Since the original offset was $f_1$, this means that if we start at a special hexagon, say $L$, and translate it through $b_1+1$ vertical columns of hexagons, always along a SW to NE diagonal of hexagons, thereby arriving in another vertical column with special hexagons, we land $f_1$ hexagons below a special hexagon. For each additional $b_1 + 1$ vertical columns of hexagons we go through in the SW to NE direction, we land an additional $f_1$ hexagons below a special hexagon. If at any moment we land a multiple of $s_1 + 1$ hexagons below a special hexagon, then we are directly on a special hexagon, since special hexagons appear every $s_1 + 1$ hexagons in the vertical column. Let $j_2$ be the smallest integer $\geq 1$ such that $j_2 \cdot f_1$ is a multiple of $s_1 + 1$ (i.e. $j_2$ is the order of $f_1$ in $Z_{s_1+1}$). Then the first time that we land directly on a special hexagon is when we have traveled through $j_2 \cdot (b_1 + 1)$ vertical columns. The string of hexagons in the SW to NE diagonal that connects the original special hexagon to this final special hexagon projects to  a spine in the quotient trihex. This spine will contain $j_2(b_1 + 1) - 1$ hexagons, since the final hexagon projects to a triangle. So this spine has  length $s_2 = j_2(b_1 + 1) - 1$. 
For example, if we start at a special hexagon of a $(5, 2, 2)$ hexagonal grid and head northeast, we will create spines of length 8, because $b_1 + 1 = 3$, $s_1 + 1 = 6$, $f_1 = 2$,  the order of $2$ in $Z_{6}$ is $j_2 = 3$, and $j_2 \cdot (b_1 + 1) - 1 = 3 \cdot 3 - 1 = 8$. See Figure~\ref{fig:SWtoNESpine}.

\begin{figure}
    \centering
    \includegraphics[height = 8 cm]{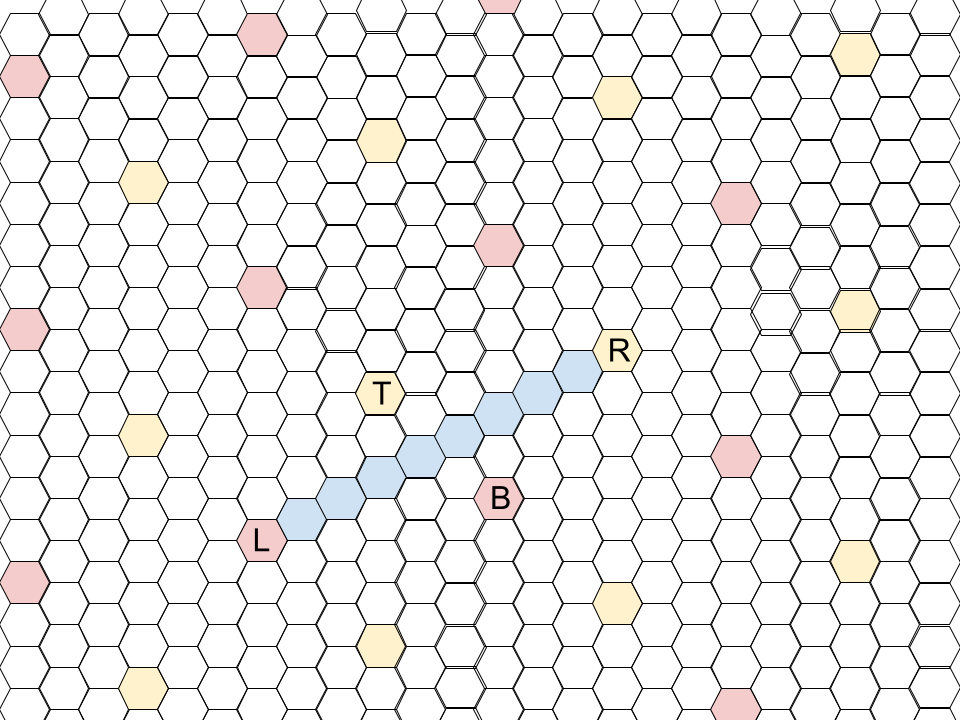}
    \caption{An alternative spine of (5,2,2) with length 8.}
    \label{fig:SWtoNESpine}
\end{figure}

Suppose instead that we translate a special hexagon in the NW to SE direction. See Figure~\ref{fig:NWtoSESpine}. If we travel through $b_1+1$ vertical columns of hexagons, thereby arriving in another vertical column with special hexagons, we land $f_1 + b_1 + 1$ hexagons below a special hexagon, since each translation through a single column in the NW to SE direction puts us one hexagon below where we would move to when translating in the SW to NE direction.  For each additional $b_1 + 1$ vertical columns of hexagons we go through in the NW to SE direction, we land an additional $f_1 + b_1 + 1$ hexagons below a special hexagon. So the first time we hit a special hexagon is when we have traveled through $j_3(b_1 + 1)$ hexagons, where $j_3$ is the smallest integer $\geq 1$ such that $j_3(f_1 + b_1 + 1)$ is a multiple of $s_1 + 1$, i.e. $j_3$ is the order of $f_1 + b_1 + 1$ in $Z_{s_1+1}$. At this point we will have created a spine of length $j_3 (b_1 + 1) - 1$. So $s_3 = j_3 (b_1 + 1) - 1$. 
For example, if we start with $(s_1, b_1, f_1) = (5, 2, 2)$, then $f_1 + b_1 + 1 = 5$ and $s_1 + 1 = 6$, and $5$ has order $j_3 = 6$ in $Z_6$. Since $j_3 \dot (b_1 + 1) - 1 = 6 \cdot 3 - 1 = 17$, we will have a spine of length 17. See Figure~\ref{fig:NWtoSESpine}.

\begin{figure}
    \centering
    \includegraphics[height = 8 cm]{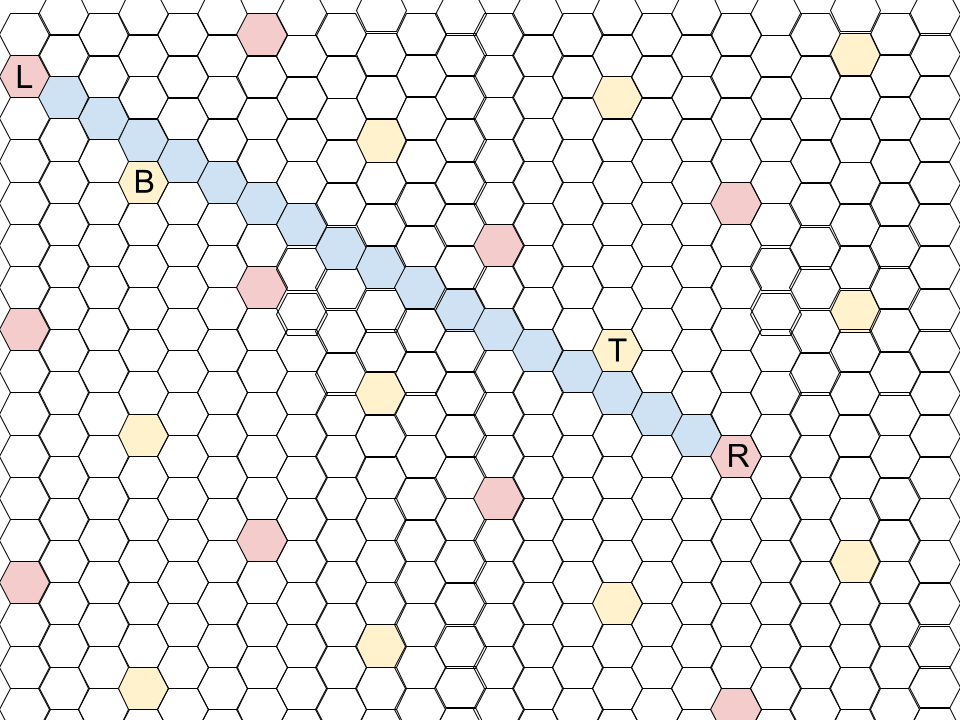}
    \caption{An alternative spine for (5,2,2) with length 17.}
    \label{fig:NWtoSESpine}
\end{figure}

To find the number of belts in the SW to NE decomposition, note that the total number $h$ of hexagons, based on the original signature of $(s_1, b_1, f_1)$, is given by $h = 2s_1 + b_1(2s_1 + 2) =2s_1b_1 + 2s_1 + 2b_1$, since each of the two spines contains $s_1$ hexagons and each of the $b_1$ surrounding belts contains $2s_1 + 2$ hexagons. If $(s_2, b_2, f_2)$ is the new signature based on spines in the SW to NE direction, then $h$ must also equal $2s_2 b_2 + 2s_2 + 2b_2$. So $b_2 = \dfrac{h - 2s_2}{2s_2 + 2}$. Similarly, the number of belts for the NW to SE decomposition with signature $(s_3, b_3, f_3)$ is given by $ b_3 = \dfrac{h - 2s_3}{2s_3 + 2}$.
For example, for the trihex $(5, 2, 2)$, we have $h = 2 \cdot 5 \cdot 2 + 2 \cdot 5 + 2 \cdot 2 = 34$. We saw that $s_2 = 8$ and $s_3 = 17$. So the number of belts using the SW to NE spines is $b_2 = \dfrac{34 - 2\cdot 8}{2 \cdot 8 + 2} = \dfrac{18}{18} = 1$. The number of belts using the NW to SE spines is $b_3 = \dfrac{34 - 2\cdot 17}{2 \cdot 17 + 2} = \dfrac{0}{36} = 0$. Note that for the SW to NE decomposition, the belts in the trihex are covered by diagonal strips of hexagons in the SW to NE direction that lie between the diagonal strips containing special hexagons. Similarly, for the NW to SE decomposition, the belts in the trihex are covered by diagonal strips of hexagons in the NW to SE direction. We will call these diagonal strips of hexagons ``belt strips''.

To find the offset for the SW to NE signature, we first need to find a special hexagon that is adjacent to the belt strips around a SW to NE spine. Label the special hexagons on the left and right ends of a fixed diagonal spine $L$ and $R$, respectively. Label the special hexagons that are adjacent to the belt strips $T$ and $B$, where $T$ is adjacent on top and $B$ is adjacent on bottom. See Figure~\ref{fig:SWtoNESpine}. Hexagons $T$ and $B$ project to the head and tail triangles of a second SW to NE spine whose position relative to the first SW to NE spine will give the offset. Hexagon $T$ will be a special hexagon that is $b_2 + 1$ hexagons above the original  diagonal spine, or equivalently, the diagonal spine is $b_2 + 1$ hexagons below $T$. Recall that each time we travel $b_1 +1$ columns along the diagonal spine in the SW to NE direction, we land an additional $f_1$ hexagons below a special hexagon. Therefore, we need to find a number $p_2$ such that $p_2\cdot f_1 \equiv (b_2 + 1) \mod (s_1 + 1)$, and travel $p_2 (b_1 + 1)$ columns to the right, in order to land in the same column as $T$, but $b_2 + 1$ hexagons below it. For simplicity, pick $p_2$ to the smallest number $\geq 1$ such that $p_2\cdot f_1 \equiv (b_2 + 1) \mod (s_1 + 1)$.

To find the corresponding offset, notice that since $T$ is $p_2 (b_1 + 1)$ columns to the right of $L$'s column, it will be $(s_2 + 1) - p_2 (b_1 + 1)$ columns to the left of $R$'s column.  If $b_2 = 0$, the number of columns to the left of $R$ is one more than the offset, so $f_2 = (s_2 + 1) - p_2 (b_1 + 1) - 1 = s_2 - p_2(b_1 +1)$. If $b_2 > 0$, then calculating offset involves deleting $b_2$ diagonal belts in the quotient trihex and moving clockwise, which is equivalent to moving $b_2$ columns to the right in the {\bf NW to SE} direction in the hexagonal tiling of the plane. Therefore, the offset will be $b_2$ smaller, that is, $f_2 = s_2 - p_2(b_1 +1) - b_2 \mod (s_2 + 1)$. Note that offset is defined mod $(s_2 + 1)$ since $s_2$ is the length of the diagonal spine. See Figure~\ref{fig:SWtoNESpine}.

For the $(5, 2, 2)$ trihex, $s_1 = 5$, $b_1 = 2$, $f_1 = 2$, $s_2 = 8$, and $b_2 = 1$. The number $p_2$ is defined as the smallest number $\geq 1$ such that $p_2 \cdot f_1 \equiv (b_2 + 1) \mod (s_1 + 1)$, i.e. such that $p_2 \cdot 2 \equiv 2 \mod 6$. Therefore, $p_2 = 1$, and $f_2 = 8 - 1\cdot 3 - 1 \mod 9 = 4$. The $(5, 2, 2)$ trihex has an alternative signature of $(8,1,4)$.

To find the offset for the NW to SE signature, label the special hexagons on the left and right ends of the diagonal spine in the NW to SE direction with $L$ and $R$, respectively, and the special hexagons that are adjacent to the surrounding belt strips $T$ and $B$, where $T$ is adjacent on top and $B$ is adjacent on bottom. See Figure~\ref{fig:NWtoSESpine}. Hexagon $T$ will be $b_3 + 1$ hexagons above the diagonal spine, or equivalently, the diagonal spine is $b_3 + 1$ hexagons below $T$. Recall that each time we travel $b_1 +1$ columns along the diagonal spine in the NW to SE direction, we land an additional $f_1 + b_1 + 1$ hexagons below a special hexagon. Therefore, we need to find a number $p_3$ such that $p_3\cdot (f_1 + b_1 + 1) \equiv (b_3 + 1) \mod (s_1 + 1)$, and travel $p_3 (b_1 + 1)$ columns to the right, in order to land in the same column as $T$, but $b_3 + 1$ hexagons below it. For simplicity, pick $p_3$ to the smallest number $\geq 1$ such that $p_3\cdot (f_1 + b_1 + 1) \equiv (b_3 + 1) \mod (s_1 + 1)$. 

To find the corresponding offset, notice that since $T$ is $p_3 (b_1 + 1)$ columns to the right of $L$'s column, it will be $(s_3 + 1) - p_3 (b_1 + 1)$ columns to the left of $R$'s column. If $b_3 = 0$, the number of columns to the left of $R$ is equal to the offset, instead of one more than the offset, like it was for the SW to NE spine. So $f_3 = (s_3 + 1) - p_3 (b_1 + 1)$. If $b_3 > 0$, then calculating offset in the quotient trihex involves deleting $b_3$ diagonal belts and moving  clockwise, which is equivalent to simply moving the hexagon $T$ straight down in its column in the hexagonal grid covering. Therefore, the offset will be $f_3 = (s_3 + 1) - p_3(b_1 +1) \mod (s_3 + 1)$. Again, the offset is defined mod $s_3 + 1$ since $s_3$ is the length of the diagonal spine. See Figure~\ref{fig:NWtoSESpine}.

For the $(5, 2, 2)$ trihex, $s_1 = 5$, $b_1 = 2$, $f_1 = 2$, $s_3 = 17$, and $b_3 = 0$, so $f_1 + b_1 + 1 = 5$. The number $p_3$ is defined as the smallest number $\geq 1$ such that $p_3 \cdot (f_1 + b_1 + 1) \equiv b_3 + 1 \mod (s_1 + 1)$, i.e. such that $p_3 \cdot 5 \equiv 1 \mod 6$. Therefore, $p_3 = 5$, and $f_3 = 18 - 5\cdot 3 \mod 18 = 3$. The $(5, 2, 2)$ trihex has an alternative signature of $(17,0,3)$.

\begin{definition}
    Given a trihex with signature $(s_1, b_1, f_1)$, the \emph{equivalent signatures} for this trihex are the original signature $(s_1, b_1, f_1)$ along with signatures $(s_2, b_2, f_2)$ and $(s_3, b_3, f_3)$ found using the following algorithms.
\label{def:altSigs}
\end{definition}

Using the SW to NE spine :

\begin{enumerate}
    \item Find the smallest number $j_2 \geq 1$ such that $j_2 \cdot f_1 \equiv 0 \mod (s_1 + 1)$.
    \item $s_2 = j_2(b_1 + 1) - 1$.
    \item Compute the total number of hexagons in the original trihex: $h = 2s_1\cdot b_1+2s_1+2b_1$. 
    \item $b_2 = \dfrac{h-2s_2}{2s_2+2}$.
    \item Find the smallest number $p_2 \geq 1$ such that $p_2\cdot f_1 \equiv (b_2 + 1) \mod (s_1 + 1)$.
    \item $f_2 = s_2 - p_2(b_1+1)-b_2 \mod (s_2+1)$.
\end{enumerate}

Using the NW to SE spine (only steps 1, 5, and 6 are different):

\begin{enumerate}
    \item Find the smallest number $j_3 \geq 1$ such that $j_3(f_1 + b_1 + 1) \equiv 0 \mod (s_1 + 1)$.
    \item $s_3 = j_3(b_1 + 1) - 1$.
    \item Compute the total number of hexagons in the original trihex: $h = 2s_1\cdot b_1+2s_1+2b_1$. 
    \item $b_3 = \dfrac{h-2s_3}{2s_3+2}$.
    \item Find the smallest number $p_3 \geq 1$ such that $p_3\cdot (f_1 + b_1 + 1) \equiv (b_3 + 1) \mod (s_1 + 1)$.
    \item $f_3 = s_3 + 1 - p_3(b_1+1) \mod (s_3+1)$.
\end{enumerate}

The signatures $(s_1, b_1, f_1)$, $(s_2, b_2, f_2)$, and $(s_3, b_3, f_3)$ give the three alternative descriptions of the same arrangement of special hexagons on a hexagonal tiling of the plane, found by rotating the ``vertical'' direction by 0 or 180 degree, 60 or 240 degrees, and 120 or 300 degrees clockwise, respectively. Therefore, this definition of equivalent signatures does in fact describe an equivalence relationship.
Although the three signatures are usually distinct, it is possible for all three to be the same. See Table~\ref{tbl:alternativeSignatures}. It is not possible for two of the three signatures to be the same and the third signature different: if two signatures are the same, say $(s_1, b_1, f_1)$ = $(s_2, b_2, f_2)$, then rotating the hexagonally tiled plane by 60 degrees will  produce the same configuration of special hexagons. Therefore, rotating a second time by the 60 degrees will again produce the same configuration of special hexagons, so $(s_3, b_3, f_3)$ will also be the same. 

Recall that two trihexes considered equivalent if they are not only  isomorphic as graphs, but if there is also an orientation-preserving homeomorphism of the plane that takes one graph to the other. Chiral trihexes that are mirror images of each other are not considered equivalent. 
Figure~\ref{fig:mirrorImage} illustrates the following relationship:

\begin{proposition}
\label{lemma:mirrorSig}
A trihex with signature $(s, b, f)$ has a mirror image trihex with signature $(s, b, s-f-b \mod (s+1))$. 
\end{proposition}

\begin{proof}
The mirror image of trihex $(s, b, f)$ will still have the same spine lengths as the original ($s$) and the same number of belts in between them ($b$). See Figure~\ref{fig:mirrorImage}

Suppose that $b = 0$. If the original trihex has offset $f$, then its mirror image will have offset $s - f$. 
Suppose $b > 0$. Since the offset of the original trihex is $f$, if we  delete the $b$ belts one at a time  and shift the head vertex clockwise each time, then the head vertex lands at offset $f$. Therefore, in the mirror image trihex, if we delete belts one at a time and shift the head vertex \emph{counterclockwise} each time, then the head vertex will land at the mirror image position offset $s-f$. Since moving counterclockwise instead of clockwise increases offset by 1 for each belt that is removed, the actual offset for the mirror image found by shifting clockwise will be $b$ less than $s-f$, i.e. $s - f - b \mod(s + 1)$.
\end{proof}

\begin{figure}

           \begin{subfigure}{.4\textwidth}
    \centering
    \includegraphics[height=5cm]{trihex3-1-2.png}

    \end{subfigure}
           \begin{subfigure}{.4\textwidth}
    \centering
    \includegraphics[height=5cm]{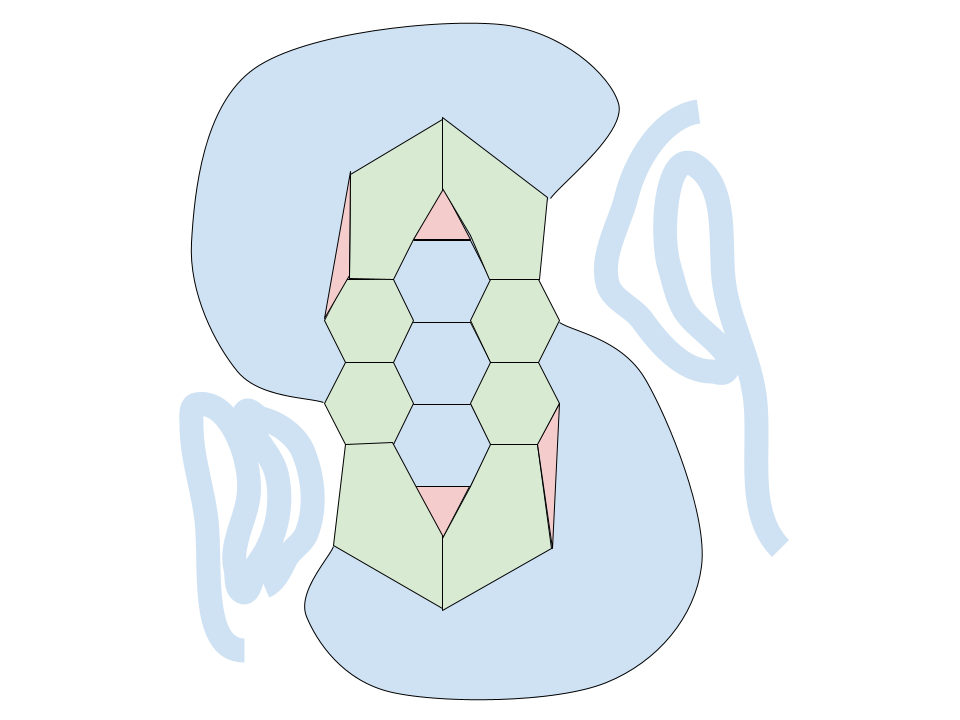}
    \end{subfigure}
    \caption{The trihex $(3, 1, 2)$ and its mirror image $(3, 1, 0)$}
 \label{fig:mirrorImage}
\end{figure}

We can now state our classification of trihexes:  trihexes are precisely indexed by the equivalence classes of triples $(s, b, f)$ under the relations for triples stated in Definition~\ref{def:altSigs}.

\begin{theorem}

\begin{enumerate}[(i)]

\item Suppose $T_0$ and $T_1$ are trihexes with signatures $(s_0, b_0, f_0)$ and $(s_1, b_1, f_1)$. Then $T_0$ and $T_1$ are equivalent trihexes if and only if $(s_0, b_0, f_0)$ and $(s_1, b_1, f_1)$ are equivalent signatures.

Therefore, there is a bijection between equivalence classes of trihexes and equivalence classes of signatures.

\item Suppose $T_0$ and $T_1$ are trihexes with signatures $(s_0, b_0, f_0)$ and $(s_1, b_1, f_1)$. Then $T_0$ and $T_1$ are isomorphic as graphs but not equivalent (i.e. they are mirror images of each other), if and only if  $(s_1, b_1, f_1)$ is  equivalent to $(s_0, b_0, s_0-f_0-b_0 \mod (s_0+1))$.

Therefore there is a bijection between graph isomorphism classes of trihexes and sets of signatures that are either equivalent or mirror equivalent.

\end{enumerate}
\label{lemma:mirrorImages}
\end{theorem}

\begin{proof}
Suppose the two trihexes $T_0$ and $T_1$ are isomorphic as graphs.
By a theorem of Whitney (\cite{whitney19332}, or see \cite{mohar2001graphs}), there is a homeomorphism of the sphere whose restriction to $T_0$ gives a graph isomorphism to $T_1$. Lift this homeomorphism to a map from the hexagonal tiling that covers $T_0$ to the hexagonal tiling that covers $T_1$.  This lifted map is a  homeomorphism of the hexagonally tiled plane that takes hexagons to hexagons and special hexagons to special hexagons. There is a unique isometry of the plane that agrees with the homeomorphism on all the vertices of the hexagonal tiling. The isometry is orientation preserving if and only if the original homeomorphism is.

If the isometry is orientation preserving, then it must be either a rotation by a multiple of $60^\circ$, or a translation, since these are the only isometries of the plane that preserve the hexagonal grid. A rotation by $180^\circ$ or $0^\circ$ or a translation takes vertical spines to vertical spines. A rotation by $60^\circ$ or $240^\circ$ counterclockwise takes vertical spines to NW to SE spines, and a rotation by $120^\circ$ or $300^\circ$ counterclockwise takes vertical spines to SW to NE spines. Therefore, $(s_0, f_0, b_0)$ must be equivalent to $(s_1, f_1, b_1)$. 

If the isometry is orientation reversing, then reflect the first hexagonal tiling through a vertical line centered at special hexagons. Consider the isometry from this reflected hexagonal tiling to the second hexagonal tiling formed as the composition of the  reflection followed by the original isometry. This composition gives an  orientation preserving isometry  from the mirror image of the first hexagonal tiling to the second hexagonal tiling. Therefore, the reflected hexagonal tiling, whose signature is $(s_0, b_0, s_0 - f_0 - b_0 \mod(s_0 + 1))$, has signature equivalent to $(s_1, b_1, f_1)$..

Conversely, suppose the signature $(s_1, b_1, f_1)$ is equivalent to the signature $(s_0, b_0, f_0)$. By  Theorem~\ref{thm:allTrihexesComeFromTilings}, both trihexes $T_0$ and $T_1$ arise as quotients of  hexagonal tilings under groups of isometries generated by $180^\circ$ rotations. Since $(s_0, b_0, f_0)$ and $(s_1, b_1, f_1)$ are equivalent signatures, the rotocenters of these rotations are the same, and so these isometry groups are the same. Therefore, the trihexes must be equivalent. If $(s_1, b_1, f_1)$ is equivalent to $(s_0, b_0, s_0 - f_0 - b_0 \mod (s_0+1))$, then the  grids of rotocenters for these rotations are mirror images of each other. Therefore, there is a reflection that takes one hexagonal tiling to the other,  takes special hexagons to special hexagons, and is preserved by the action of the rotation groups. This reflection projects to an orientation-reversing homeomorphism between quotient spheres that is a graph isomorphism between $T_0$ and $T_1$. 

The existence of bijections now follows from Theorem~\ref{thm:achieveAllSigs}, which says that every signature is realized by a trihex and every trihex has a signature.
\end{proof}

Table~\ref{tbl:alternativeSignatures} gives the signatures for all trihexes with 20 hexagons or fewer (44 vertices or fewer). Each row give the three equivalent signatures for a trihex. The three signatures are ordered so that the signature with the smallest value of $b$ is on the left, with preference given to the signature with smaller value of $f$ in case of a tie. In some cases, the alternative signatures are repetitions. For example, the alternative signatures for $(6, 0, 2)$ are $(6, 0, 2)$ and $(6, 0, 2)$.

\begin{table}
\begin{tabular}{|c|c|c|c|c|} \hline
$(s_1, b_1, f_1)$ & $(s_2, b_2, f_2)$ & $(s_3, b_3, f_3)$ & hexagons & vertices \\ \hline
(0, 0, 0) & (0, 0, 0) & (0, 0, 0) & 0 & 4 \\ 
(1, 0, 0) & (1, 0, 1) & (0, 1, 0) & 2 & 8 \\ 
(2, 0, 0) & (2, 0, 2) & (0, 2, 0) & 4 & 12 \\ 
(2, 0, 1) & (2, 0, 1) & (2, 0, 1) & 4 
 & 12\\ 
(3, 0, 0) & (3, 0, 3) & (0, 3, 0) & 6 & 16 \\ 
(3, 0, 1) & (3, 0, 2) & (1, 1, 1) & 6 & 16\\ 
(1, 1, 0) & (1, 1, 0) & (1, 1, 0) & 6 & 16\\ 
(4, 0, 0) & (4, 0, 4) & (0, 4, 0) & 8 & 20 \\ 
(4, 0, 1) & (4, 0, 2) & (4, 0, 3) & 8 & 20 \\ 
(5, 0, 0) & (5, 0, 5) & (0, 5, 0) & 10 & 24\\ 
(5, 0, 1) & (5, 0, 4) & (2, 1, 2) & 10 & 24\\ 
(5, 0, 2) & (2, 1, 0) & (1, 2, 1) & 10 & 24\\ 
(5, 0, 3) & (2, 1, 1) & (1, 2, 0) & 10 & 24\\ 
(6, 0, 0) & (6, 0, 6) & (0, 6, 0) & 12 & 28\\ 
(6, 0, 1) & (6, 0, 3) & (6, 0, 5) & 12 & 28\\ 
(6, 0, 2) & (6, 0, 2) & (6, 0, 2) & 12 & 28\\ 
(6, 0, 4) & (6, 0, 4) & (6, 0, 4) & 12 & 28\\ 
(7, 0, 0) & (7, 0, 7) & (0, 7, 0) & 14 & 32\\ 
(7, 0, 1) & (7, 0, 6) & (3, 1, 3) & 14 & 32\\ 
(7, 0, 2) & (7, 0, 5) & (3, 1, 1) & 14 & 32\\ 
(7, 0, 3) & (7, 0, 4) & (1, 3, 1) & 14 & 32\\ 
(3, 1, 0) & (3, 1, 2) & (1, 3, 0) & 14 & 32\\ 
(8, 0, 0) & (8, 0, 8) & (0, 8, 0) & 16 & 36\\ 
(8, 0, 1) & (8, 0, 4) & (8, 0, 7) & 16 & 36\\ 
(8, 0, 2) & (8, 0, 3) & (2, 2, 2) & 16 & 36\\ 
(8, 0, 5) & (8, 0, 6) & (2, 2, 1) & 16 & 36\\ 
(2, 2, 0) & (2, 2, 0) & (2, 2, 0) & 16 & 36\\ 
(9, 0, 0) & (9, 0, 9) & (0, 9, 0) & 18 & 40\\ 
(9, 0, 1) & (9, 0, 8) & (4, 1, 4) & 18 & 40\\ 
(9, 0, 2) & (9, 0, 3) & (4, 1, 2) & 18 & 40\\ 
(9, 0, 4) & (4, 1, 0) & (1, 4, 1) & 18 & 40\\ 
(9, 0, 5) & (4, 1, 3) & (1, 4, 0) & 18 & 40\\ 
(9, 0, 6) & (9, 0, 7) & (4, 1, 1) & 18 & 40\\ 
(10, 0, 0) & (10, 0, 10) & (0, 10, 0) & 20 & 44\\ 
(10, 0, 1) & (10, 0, 5) & (10, 0, 9) & 20 & 44\\ 
(10, 0, 2) & (10, 0, 4) & (10, 0, 7) & 20 & 44\\ 
(10, 0, 3) & (10, 0, 6) & (10, 0, 8) & 20 & 44\\  \hline
\end{tabular}
\caption{The three equivalent signatures for trihexes with 20 or fewer hexagons / 44 or fewer vertices.}
\label{tbl:alternativeSignatures}
\end{table}

\section{Convex vs. non-convex trihexes}
\label{sec:convexVsNonconvex}

The trihex signature determines whether it can arise from a convex polyhedron. We will first quote some preliminary facts. 

\begin{proposition}
\label{lemma:facts}
\begin{enumerate}[(i)]
\item Every trihex is 2-connected.
\item If a trihex is not 3-connected, then it is a godseye.
\item A trihex is a simple graph. 
\end{enumerate}
\end{proposition}

\begin{proof}

Parts (i) and (ii) are proved in \cite{deza2005zigzag}.




Part (iii): If a trihex had an edge loop, then it would have a face with only one edge. If it had a double edge, then it would either have a face with only two edges or else each of the vertices on the double edge would be a separating vertex, contradicting Part (i).
\end{proof}

\begin{theorem}
Any trihex with a signature $(0, b, 0)$ with $b > 0$ can be represented as the skeleton of a non-convex polyhedron, and it cannot be represented as the skeleton of a convex polyhedron. All other trihexes can be represented as skeletons of convex polyhedra.

\end{theorem}

\begin{proof}
Steinitz's Theorem \cite{grunbaum1967convex} or \cite{ziegler2012lectures}, says that a graph can be represented as the skeleton of a convex polyhedron if and only if the graph is simple, planar, and 3-connected. Trihexes with signature $(0, b, 0)$ with $b > 0$ are godseyes, which are not 3-connected: removing the two vertices shown in red in Figure~\ref{fig:godseye} disconnects the graph. Therefore they are not the skeletons of convex polyhedra. However, they are the skeletons of non-convex polyhedra, as shown in Figure~\ref{fig:nonConvexPolyhedra}.
All other trihexes are 3-connected, simple, planar graphs by Lemma~\ref{lemma:facts}. So by Steinitz's theorem, they are the skeletons of convex polyhedra.
\end{proof}

\begin{figure}
    \begin{subfigure}{.5\textwidth}
        \centering
    \includegraphics[height=6cm]{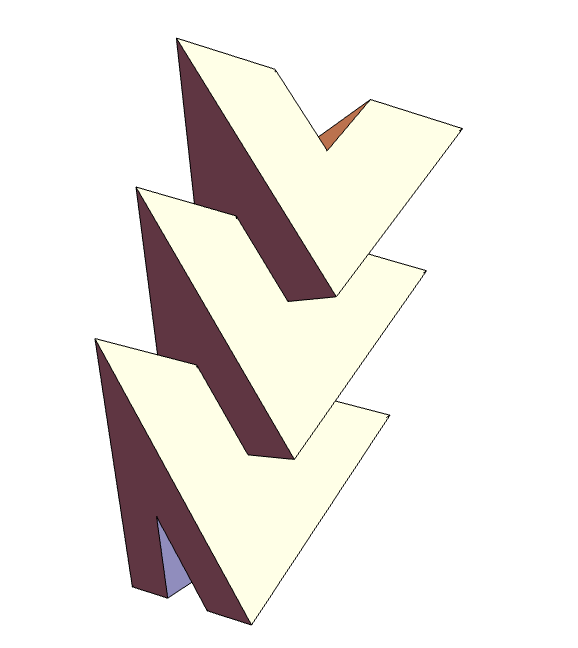}
    \caption{Godseye with six pairs of hexagons}
    \end{subfigure}
    \begin{subfigure}{.5\textwidth}
        \centering
    \includegraphics[height=6cm]{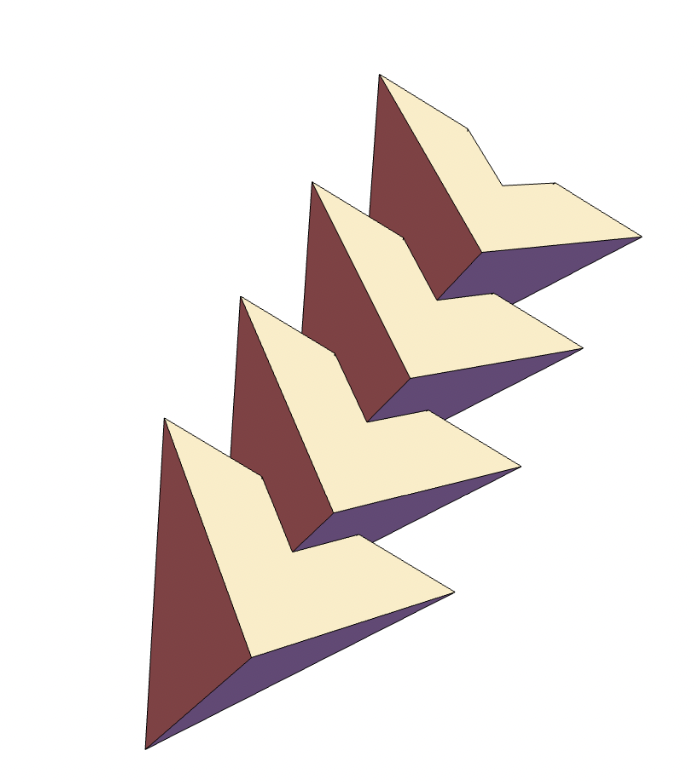}
    \caption{Godseye with seven pairs of hexagons}
    \end{subfigure}
\caption{Godseyes can be realized as non-convex polyhedra.}
\label{fig:nonConvexPolyhedra}
\end{figure}

\section{How many trihexes of each size?}

\label{sec:howMany}

Given $v \geq 0$, how many trihexes are there with $v$ vertices? Let $\alpha(v)$ be the number of equivalence classes of trihexes with $v$ vertices and let $\beta(v)$ be the number of graph isomorphism classes of trihexes with $v$ vertices. These two quantities can be computed by counting signatures and accounting for duplicates, after establishing the following relationships, which are also evident in \cite{grunbaum_motzkin_1963}.

\begin{lemma}
\begin{enumerate}[(i)]
    \item The triple $(s, b, f)$ is a signature for a trihex with $h$ hexagons if and only if $s \geq 0$, $b \geq 0$, $0 \leq f \leq s$, and $\dfrac{h}{2} + 1 =(s+1)(b+1)$.

    \item The triple $(s, b, f)$ is a signature for a trihex with $v$ vertices if and only if $s \geq 0$,  $b \geq 0$, $0 \leq f \leq s$, and $\dfrac{v}{4} = (s+1)(b+1)$.

    \item Consequently, the number of hexagons in a trihex is even and the  number of vertices of a trihex is divisible by 4. 
\end{enumerate}
\end{lemma}

\begin{proof}
A trihex with signature $(s, b, f)$ has $2s + b(2s+2) = 2(s+1)(b+1) -2 $ hexagons, since each of the two spines contains $s$ hexagons and each of the $b$ belts contains $2s+2$ hexagons. So for $s \geq 0$, $b \geq 0$, $0 \leq f \leq s$, the triple $(s, b, f)$ is a signature for a trihex with $h$ hexagons if and only if $\dfrac{h}{2} = (s + 1)(b+ 1)  - 1$. Since each hexagon has six vertices and each of the four triangles in a trihex has four vertices, the number of vertices in a trihex is $v = \dfrac{6h + 12}{3} = 2h + 4 = 4(s+1)(b+1)$. So the triple $(s, b, f)$ with $s \geq 0$, $b \geq 0$, and $0 \leq f \leq s$ is a signature for a trihex with $v$ vertices if and only if $\dfrac{v}{4} = (s+1)(b+1)$. Since $\dfrac{h}{2}$ and $\dfrac{v}{4}$ are integers, $h$ is divisiby by 2 and $v$ is divisible by 4.
\end{proof}

From this relationship, we can compute $\alpha(v)$ and $\beta(v)$ as follows. For each desired value of $v$, we can find the list of all triples $(s, b, f)$ satisfying the inequalities $s \geq 0$, $b \geq 0$, and $0 \leq f \leq s$ and the equation $\dfrac{v}{4} = (s+1)(b+1)$. 
Then we can test which triples satisfy the relationships in Definition~\ref{def:altSigs} and therefore represent equivalent trihexes. Taking these equivalences into account yields $\alpha(v)$. Checking for triples that satisfy the mirror equivalence relation in Theorem~\ref{lemma:mirrorImages} yields $\beta(v)$. See Table~\ref{tbl:numberOfTrihexes} for counts of $\alpha(v)$ and $\beta(v)$ for $v \leq 200$.

Let $\sigma(v)$ be the number of triples that form a signature for a trihex with $v$ vertices; that is, the number of triples $(s, b, f)$ with $s \geq 0$, $b \geq 0$, $0 \leq f \leq s$, and $v = 4(s+1)(b+1)$. 

\begin{lemma}
    Let $p_1^{m_1} p_2^{m_2} \cdot p_k^{m_k}$ be the prime factorization of $\dfrac{v}{4}$. Then $\displaystyle{\sigma(v) = \prod_{i = 1}^k \dfrac{p_i^{m_i + 1} - 1}{p_i - 1}}$.
\end{lemma}

\begin{proof}
    The set of pairs $(s, b)$ that satisfy (i) $s \geq 0$. (ii) $b \geq 0$,  and (iii) $v = 4(s+1)(b+1)$ is in one-to-one correspondence with the factors of $\dfrac{v}{4}$, by the correspondence that takes a factor $d$ to the pair $(s, b) = \left(d-1, \dfrac{v}{4d} - 1\right)$. For each such pair $(s, b)$, corresponding to the factor $d = s+1$ of $\dfrac{v}{4}$, there are $s+1$ triples $(s, b, f)$ that satisfy (iv) $0 \leq f \leq s$. Therefore, the number of triples $(s, b, f)$ that satisfy (i), (ii), (iii), and (iv) is equal to the sum of the factors $d$ of $\dfrac{v}{4}$. This sum is equal to $\displaystyle{\prod_{i = 1}^k \dfrac{p_i^{m_i + 1} - 1}{p_i - 1}}$.
\end{proof}

\begin{proposition}
Let $p_1^{m_1} p_2^{m_2} \cdot p_k^{m_k}$ be the prime factorization of $\dfrac{v}{4}$. Then

$$\dfrac{1}{3} \prod_{i = 1}^k \dfrac{p_i^{m_i + 1} - 1}{p_i - 1} \leq \alpha(v) \leq \prod_{i = 1}^k \dfrac{p_i^{m_i + 1} - 1}{p_i - 1}$$ and 
    $$\dfrac{1}{6} \prod_{i = 1}^k \dfrac{p_i^{m_i + 1} - 1}{p_i - 1} \leq \beta(v) \leq \prod_{i = 1}^k \dfrac{p_i^{m_i + 1} - 1}{p_i - 1}$$
\end{proposition}

\begin{proof}

If each of the triples $(s, b, f)$ with $s \geq 0$, $b \geq 0$, $0 \leq f \leq s$, and $(s+1)(b+1) = \dfrac{v}{4}$  represented a distinct trihex, there would be $\sigma(v)$ distinct trihexes with $v$ vertices. However,  some of these triples represent equivalent trihexes, since a trihex can be decomposed in three ways into spines. Therefore $\dfrac{\sigma(v)}{3} \leq \alpha(v) \leq \sigma(v)$. There could be up to six signatures that represent isomorphic trihexes, since the left-handed and right-handed versions of chiral trihexes each have their own three signatures.   So
$\dfrac{\sigma(v)}{6} \leq \beta(v) \leq \sigma(v)$.  

\end{proof}

Table~\ref{tbl:numberOfTrihexes} gives the number $\alpha(v)$ of equivalence classes of trihexes and the number $\beta(v)$ of graph isomorphism classes of trihexes for each number $v$ of vertices for $0 \leq v \leq 200$. Figure~\ref{fig:trihexCount} presents the same information graphically for $0 \leq v \leq 400$. Note that the counts in the $\beta(v)$ column of Table~\ref{tbl:numberOfTrihexes} are always one greater than the counts in Table 5 of \cite{deza2005zigzag}, since our counts include non-convex godseyes, and there is exactly one godseye for each possible number of vertices. 

For each line in  Table~\ref{tbl:numberOfTrihexes}, the count $\alpha(v)$ is close to $\left \lceil \dfrac{\sigma(v)}{3} \right\rceil$ and $\beta(v)$ is fairly close to $\left \lceil\dfrac{\sigma(v)}{6}\right \rceil$. For $200 \leq v \leq 4000$, the difference $\alpha(v) - \left\lceil \dfrac{\sigma(v)}{3} \right\rceil$ is greater than one only 2.5\% of the time and is never more than four. For $200 \leq v \leq 4000$, $\left(\dfrac{\sigma(v)}{3} \right) \leq \alpha(v) \leq 1.06 \left( \dfrac{\sigma(v)}{3} \right)$.  For $200 \leq v \leq 4000$, the difference $\beta(v) - \left \lceil \dfrac{\sigma(v)}{6} \right \rceil$ is greater than one 36.2\% of the time and has a maximum value of 22. For $200 \leq v \leq 4000$, $\left( \dfrac{\sigma(v)}{6} \right) \leq \beta(h) \leq 1.25 \left( \dfrac{\sigma(v)}{6} \right) $.

\begin{conjecture}
As $v \to \infty$, the counts $\alpha(v)$ and $\beta(v)$ are respectively asymptotic to $\dfrac{1}{3}\sigma(v)$ and $\dfrac{1}{6} \sigma(v)$.
\end{conjecture}

\begin{table}
\begin{tabular}{|ccccc||ccccc|} \hline
$v$ & $\alpha(v)$ & $\lceil \frac{\sigma(v)}{3} \rceil$ &  $\beta(v)$ & $\lceil \frac{\sigma(v)}{6} \rceil$ & $v$ & $\alpha(v)$ & $\lceil \frac{\sigma(v)}{3} \rceil$ &  $\beta(v)$ & $\lceil \frac{\sigma(v)}{6} \rceil$\\ \hline
4	&	1	&	1	&	1	&	1	&	104	&	14	&	14	&	8	&	7 \\
8	&	1	&	1	&	1	&	1	&	108	&	14	&	14	&	9	&	7	\\
12	&	2	&	2	&	2	&	1	&	112	&	20	&	19	&	13	&	10	\\
16	&	3	&	3	&	3	&	2	&	116	&	10	&	10	&	6	&	5	\\
20	&	2	&	2	&	2	&	1	&	120	&	24	&	24	&	14	&	12	\\
24	&	4	&	4	&	3	&	2	&	124	&	12	&	11	&	7	&	6	\\
28	&	4	&	3	&	3	&	2	&	128	&	21	&	21	&	15	&	11	\\
32	&	5	&	5	&	5	&	3	&	132	&	16	&	16	&	10	&	8	\\
36	&	5	&	5	&	4	&	3	&	136	&	18	&	18	&	10	&	9	\\
40	&	6	&	6	&	4	&	3	&	140	&	16	&	16	&	10	&	8	\\
44	&	4	&	4	&	3	&	2	&	144	&	31	&	31	&	20	&	16	\\
48	&	10	&	10	&	8	&	5	&	148	&	14	&	13	&	8	&	7	\\
52	&	6	&	5	&	4	&	3	&	152	&	20	&	20	&	11	&	10	\\
56	&	8	&	8	&	5	&	4	&	156	&	20	&	19	&	12	&	10	\\
60	&	8	&	8	&	6	&	4	&	160	&	30	&	30	&	20	&	15	\\
64	&	11	&	11	&	9	&	6	&	164	&	14	&	14	&	8	&	7	\\
68	&	6	&	6	&	4	&	3	&	168	&	32	&	32	&	18	&	16	\\
72	&	13	&	13	&	8	&	7	&	172	&	16	&	15	&	9	&	8	\\
76	&	8	&	7	&	5	&	4	&	176	&	28	&	28	&	17	&	14	\\
80	&	14	&	14	&	10	&	7	&	180	&	26	&	26	&	16	&	13	\\
84	&	12	&	11	&	8	&	6	&	184	&	24	&	24	&	13	&	12	\\
88	&	12	&	12	&	7	&	6	&	188	&	16	&	16	&	9	&	8	\\
92	&	8	&	8	&	5	&	4	&	192	&	42	&	42	&	28	&	21	\\
96	&	20	&	20	&	15	&	10	&	196	&	21	&	19	&	12	&	10	\\
100	&	11	&	11	&	7	&	6	&	200	&	31	&	31	&	17	&	16		\\ \hline
\end{tabular}
\caption{Count $\alpha(v)$ of equivalence classes of trihexes  and count $\beta(v)$ of graph isomorphism classes of trihexes.  Here $\sigma(v)$ is the sum of the factors of $\dfrac{v}{4}$.}
\label{tbl:numberOfTrihexes}
\end{table}

\begin{figure}
    \centering
    \includegraphics[height = 8 cm]{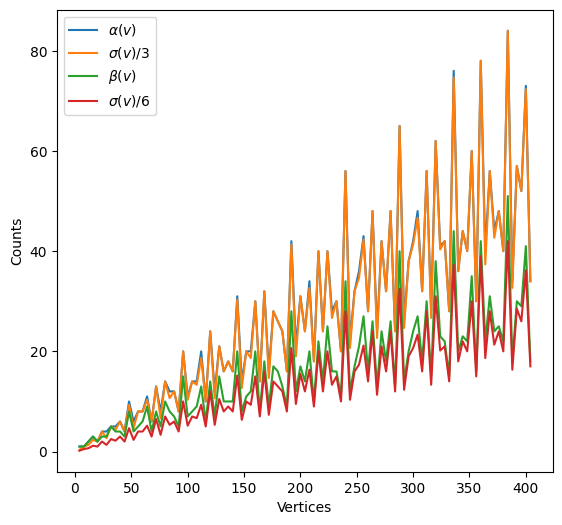}
    \caption{The number of distinct trihexes that contain $v$ vertices.}
    \label{fig:trihexCount}
\end{figure}

\section{Tight trihexes}
\label{sec:tight}

This section analyzes trihexes that contain no belts. We start with a specialization of a definition from \cite{deza2005zigzag}.

\begin{definition}
    A trihex is \emph{tight} if it does not contain any belts. 
\end{definition}

\begin{proposition} 
The following are equivalent for a trihex:

\begin{enumerate}[(i)]
\item The trihex is tight. 
\item The second coordinate in each of its three signatures is 0. 
\item For any fixed signature $(s_1, b_1, f_1)$, the three numbers $f_1$, $f_1+1$, and $s_1+1$ are pairwise relatively prime and $b_1 = 0$.
\end{enumerate}
\label{lemma:tight}
\end{proposition}

\begin{proof}
$(i) \iff (ii)$: Suppose the trihex is tight. Then the second coordinates in each of its three signatures must be zero, since the second coordinate indicates the number of belts between spines. Conversely, if the second coordinate in each of its three signatures is 0, then in the hexagonal tiling that covers the trihex, there are no infinite strips of hexagons that do not contain special hexagons, either in the vertical direction, the NW to SE direction, or the SW to NE direction. But any belt in the trihex would lift to such a strip, so the trihex must be tight.

$(ii) \implies (iii)$: Consider a fixed signature $(s_1, b_1, f_1)$ and the alternate signatures $(s_2, b_2, f_2)$ and $(s_3, b_3, f_3)$. Since the number of hexagons $h = 2s_i b_i + 2s_i + 2b_i$ for $i = 1, 2, 3$, and $b_i = 0$, we have that $s_1 = s_2 = s_3$. In the algorithm given in Definition~\ref{def:altSigs}, $ s_2 = j_2(b_1 + 1) - 1$, where $j_2$ is the order of $f_1$ in $Z_{s_1 + 1}$. Since $b_1 = 0$ and $s_2 = s_1$, this means that $s_1 = j_2 - 1$, that is, $f_1$ has order $s_1 + 1$ in $Z_{s_1 + 1}$. Equivalently, $f_1$ is relatively prime to $s_1 + 1$. Similarly, $s_3 = j_3(b_1 + 1) - 1$, so $s_1 = s_3 = j_3 - 1$, where $j_3$ is the order of $f_1 + b_1 + 1$ in $Z_{s_1 + 1}$, that is, the order of $f_1 + 1$ in $Z_{s_1 + 1}$. So $f_1 + 1$ has order $s_1 + 1$ in $Z_{s_1 + 1}$ and is therefore relatively prime to $s_1 + 1$. Since $f_1$ and $f_1 + 1$ are relatively prime, we have that $f_1$, $f_1  +1$, and $s_1  + 1$ are pairwise relatively prime. 

$(iii) \implies (ii)$: Suppose $b_1 = 0$ and $f_1$, $f_1 + 1$, and $s_1 + 1$ are relatively prime. Since $b_1 = 0$, $h = 2s_1$. In Definition~\ref{def:altSigs}, since $f_1$ is relatively prime to $s_1 + 1$, $j_2 = s_1 + 1$. Therefore, $s_2 = j_2(b_1 + 1) - 1 = (s_1 + 1)(0 + 1) - 1 = s_1$, so $b_2 = \dfrac{h - 2s_2}{2s_2 + 2} = \dfrac{h - 2s_1}{2s_1 + 2} = b_1 = 0$. 

Furthermore, since $f_1 + b_1 + 1 = f_1 + 1$ is relatively prime to $s_1 + 1$, $j_3 = s_1 + 1$. So $s_3 = j_3(b_1 + 1) - 1 = (s_1 + 1)(0 + 1) - 1 = s_1$, and $b_3 = \dfrac{h - 2s_3}{2s_3 + 2} = \dfrac{h - 2s_1}{2s_1 + 2} = b_1 = 0$.

\end{proof}

Our next result proves Conjecture 5.3 from Deza and Dutour  \cite{deza2005zigzag}. We first repeat their definition of the graph of curvatures:

\begin{definition}
 The \emph{graph of curvatures} of a trihex is the graph whose vertices are the four triangular faces. Two vertices $c$ and $d$ are connected by an edge if there exists
a pseudo-road connecting the faces $c$ and $d$. A \emph{pseudo-road} is a sequence of hexagons, say
$a_1, \ldots, a_\ell$, such that setting $a_0 = c$ and $a_{\ell+1} = d$, we have that for $1 \leq i \leq \ell$,  $a_i$,is adjacent
to $a_{i-1}$ and $a_{i+1}$ on opposite edges.
\end{definition}

\begin{proposition}
The graph of curvatures of any tight trihex is a complete graph on 4 vertices. 
\end{proposition}

\begin{proof} 
Consider a tight trihex with signature $(s_1, b_1, f_1)$. Recall that the trihex arises as the quotient sphere from a hexagonal tiling of the plane, where the ``special hexagons'',  which correspond to the four triangles, lie on the vertices of a superimposed parallelogram grid. Label these special hexagons $A$, $B$, $C$, and $D$ as in Figure~\ref{fig:hexagonalGridForTightTrihex}, so that hexagons labeled $A$ and $C$ lie in vertical columns, and hexagons labeled $B$ and $D$ lie in alternate vertical columns. 

\begin{figure}
    \centering
    \includegraphics[height = 8 cm]{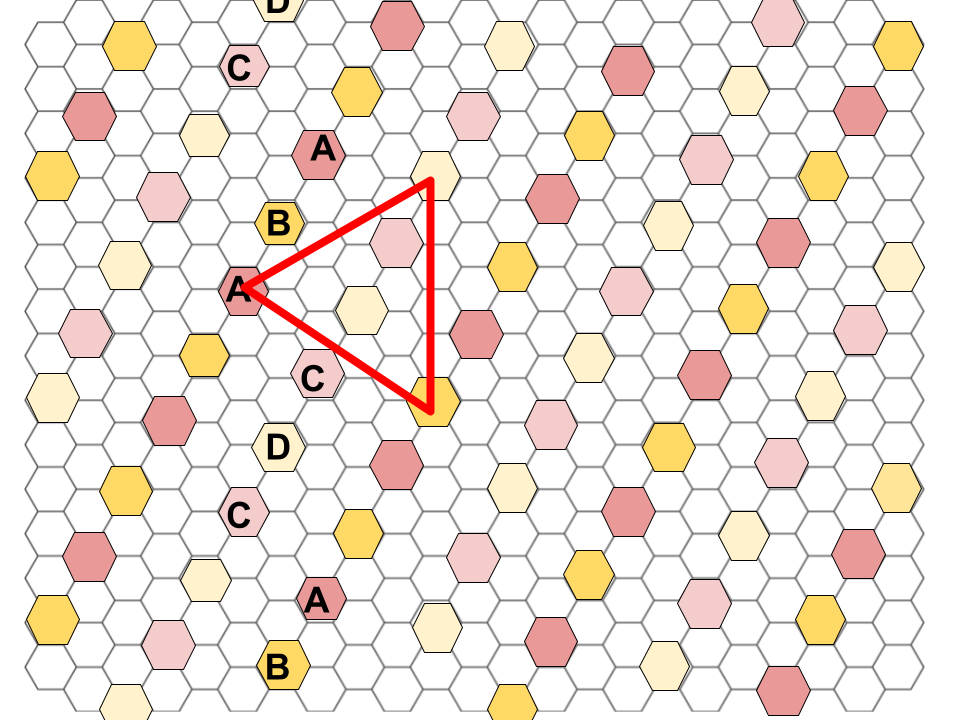}
    \caption{Hexagonal cover for a tight trihex.}
    \label{fig:hexagonalGridForTightTrihex}
\end{figure}

By Proposition~\ref{lemma:tight}, $b_1 = 0$ and $f_1$ and $f_1 + 1$ are both relatively prime to $s_1 + 1$. Therefore, $s_1 + 1$ must be odd. In Definition~\ref{def:altSigs}, $j_2$ and $j_3$, which are the orders of $f_1$ and $f_1 + 1$ in $Z_{s_1 + 1}$, respectively, must both equal $s_1 + 1$ and therefore also be odd. Since the columns of hexagons alternate between columns containing $A$ and $C$ and columns containing $B$ and $D$, if we translate a special hexagon $A$ in the SW to NE direction by $j_2 = s_1 + 1$ columns, it will hit a special hexagon in the $B/D$ column. Similarly, if we translate $A$ in by $j_3 = s_1 + 1$ columns in the NW to SE direction, it will hit a special hexagon in the $B/ D$ column. Form a triangle with vertices at the centers of the original special hexagon $A$ and these two special hexagons. The sides with a vertex in $A$ both have the same length and are at a $60^\circ$ angle from each other. Therefore, this triangle must be an equilateral triangle. So the two special hexagons in the $B/D$ column must be $s_1$ hexagons apart, so one of them is a hexagon $B$ and the other a hexagon $D$. Therefore, the graph of curvatures for the trihex must include an edge from $A$ to $B$ and an edge from $A$ to $D$.  Of course, if we translate a special hexagon $A$ straight north or south, it will hit a special hexagon $C$, so the graph of curvatures also includes an edge from $A$ to $C$. By repeating this argument starting with hexagons $B$, $C$, and $D$, in turn, instead of $A$, we see that each vertex in the graph of curvatures connects to every other vertex. 
\end{proof}

\section*{Acknowledgments}

We would like to express our appreciation to our colleague Bob Proctor for excellent suggestions and edits to an early version of this manuscript.

\section*{Statements and Declarations}

\subsection*{Funding}

The authors declare that no funds, grants, or other support were received during the preparation of this manuscript.

\subsection*{Competing Interests}
The authors have no relevant financial or non-financial interests to disclose. 

\subsection*{Author Contributions}

Both authors contributed to the research and read and approved the final manuscript.

 \subsection*{Data Availability}
 
 All data generated or analyzed during this study are included in this published article. Python code used to generate this data is available from the corresponding author on request.

\printbibliography
 
\end{document}